\documentclass[12pt]{article}
\usepackage[colorlinks,citecolor=blue,urlcolor=blue]{hyperref}
\usepackage{amsmath,amssymb,amsfonts,amsthm,graphicx}
\usepackage[usenames,dvipsnames]{color}
\newcommand{\C}{{\mathcal C}}

\newcommand{\G}{{\mathcal G}}
\newcommand{\be}{\begin{equation}}
\newcommand{\ee}{\end{equation}} 

\newcommand{\old}[1]{}
\newcommand{\eps}{\varepsilon}

\newcommand{\D}{{\cal D}}

\renewcommand{\P}{{\cal P}}

\newcommand{\R}{{\mathbb R}}
\newcommand{\Z}{{\mathbb Z}}

\newcommand{\E}{{\mathcal E}}
\newcommand{\M}{{\mathcal M}}

\newcommand{\bV}{V_\partial}
\newcommand{\bW}{W_\partial}
\newcommand{\bB}{B_\partial}

\newcommand{\SL}{\text{SL}}

\renewcommand{\H}{\mathbb{H}}
\newcommand{\Tr}{\mathrm{Tr}}

\newcommand{\rf}[1]{\begin{rotatebox}{90}{$#1$}\end{rotatebox}}

\newtheorem{theorem}{Theorem}
\newtheorem{thm}[theorem]{Theorem}

\newtheorem{lemma}[theorem]{Lemma}

\title{
Planar $3$-webs and the boundary measurement matrix}
\author{Richard Kenyon\footnote{Department of Mathematics, Yale University, New Haven; richard.kenyon at yale.edu} \and
 Haolin Shi\footnote{Department of Mathematics, Yale University, New Haven; haolin.shi at yale.edu}}

\begin{document}

\date{}
\maketitle
\abstract{
We compute connection probabilities for reduced $3$-webs in the triple-dimer model 
on circular planar graphs using the boundary measurement matrix (reduced Kasteleyn matrix).
As one application we compute several ``$SL_3$ generalizations'' of the Lindstr\o{}m-Gessel-Viennot theorem,
for ``parallel" webs and for honeycomb webs.
We also apply our results to the scaling limit of the dimer model in a planar domain, giving
conformally invariant expressions for reduced web probabilities.}

\section{Introduction}

The dimer model is the study of the set of dimer covers, or perfect matchings, of a planar graph.
It is a highly successful combinatorial and probabilistic model with connections to the Grassmannian (see e.g. \cite{Post}), conformal field theory (see e.g. \cite{Kenyon.lectures}) and integrable systems \cite{GK}. 

A \emph{double dimer cover} of a graph is obtained by superimposing two dimer covers, to get a collection of loops and doubled edges.
In \cite{KW11} Kenyon and Wilson studied \emph{boundary connection probabilities} for the double dimer model in planar graphs (see an example in Figure \ref{doubledomino}),
\begin{figure}[htbp]
\begin{center}\includegraphics[width=2.5in]{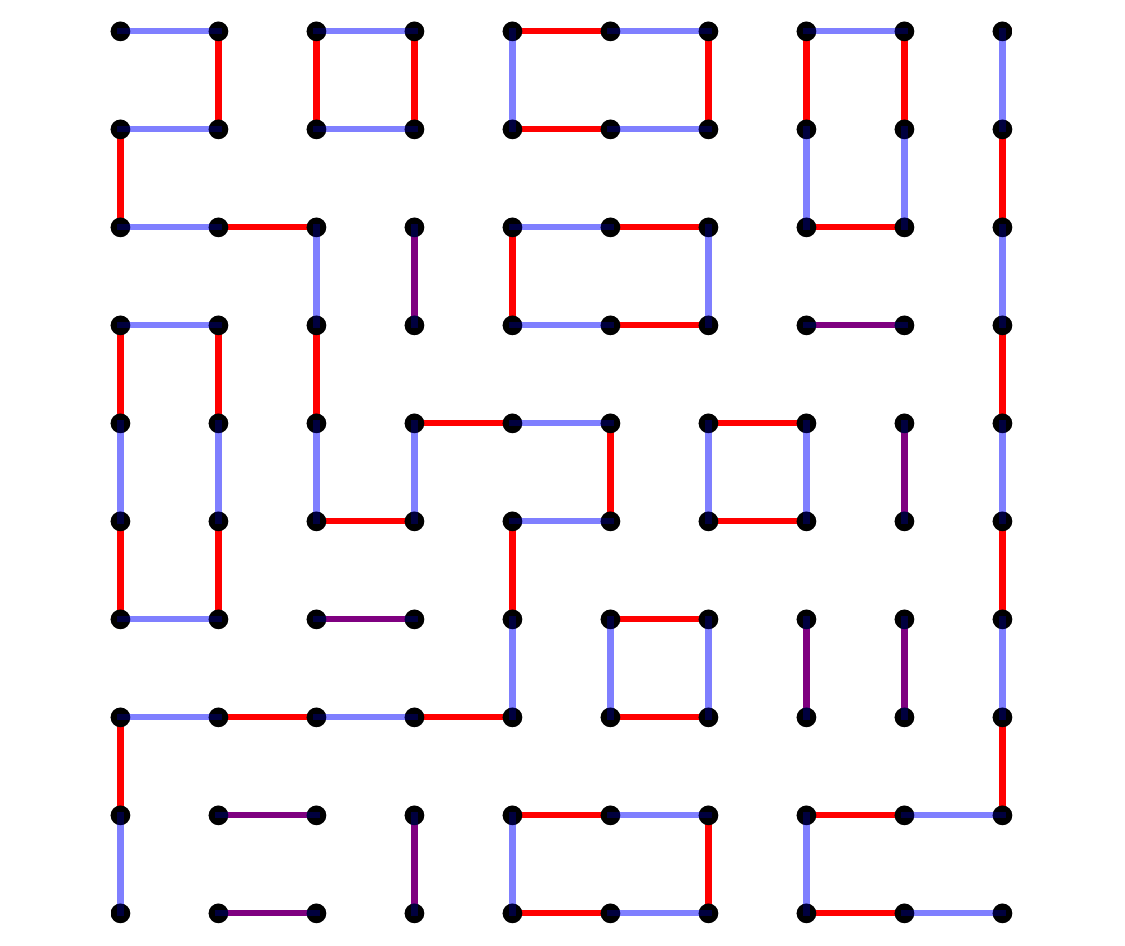}\end{center}
\caption{\label{doubledomino}The union of two independent uniform dimer covers of a rectangle (the first of 
which does not use the four corner vertices)
produces a double dimer configuration containing two chains of dimers connecting the four corners in two possible ways.
The probability of either connection is a function of the boundary measurement matrix \cite{KW11}.}
\end{figure}
showing how they could be computed from the reduced Kasteleyn matrix, or ``boundary measurement matrix".
As application, in the scaling limit they computed connection probabilities for multiple $\text{SLE}_4$ curves in the disk. Similar calculations were recently done in \cite{Peltolaetal}
for the spanning tree model and $\text{SLE}_8$ from a conformal field theory point of view.
In \cite{Kenyon14CMP, Dubedat, BC} it was shown how to use $\SL_2(\R)$-connections on graphs to compute various topological connection probabilities 
for the double dimer model on a graph on a topologically nontrivial surface without boundary. 

Given these results, it is natural to ask about possible $n$-fold dimer analogs, for $n\ge 3$. 
A union
of $n$ dimer covers makes a more complicated structure called a  \emph{graphical $n$-web}, or \emph{$n$-multiweb}, see \cite{DKS}. 
We consider the case $n=3$ here, which has an extra feature not present in the $n\ge4$ case, which is the notion of \emph{reduced} webs.  

Webs are primarily representation-theoretic objects.
For $n=3$, Kuperberg in \cite{Kuperberg} used them to describe invariant functions in tensor products of $\SL_3$-representations.
He also gave a diagrammatic basis for invariant functions consisting of reduced $3$-webs.
In \cite{Pasha1}, Pylyavskyy discussed relations between reduced $3$-webs and the combinatorics of totally positive matrices.
Lam in \cite{Lam} first discusses the connection between 
$3$-webs and the boundary measurement matrix, concentrating on the algebraic and integrable aspects, in particular for 
positroid varieties. 
Fraser, Lam, Le in \cite{FLL} also connect $n$-fold dimers with the boundary measurement matrix, studying in particular the structure
of the space of $n$-webs for $n\ge 4$. In \cite{DKS} the connection was made between $\SL_n$ web traces and the associated Kasteleyn matrix determinant,
for $n$-webs in general planar bipartite graphs.

Even though they have a representation-theoretic underpinning, webs are also natural from a combinatorial point of view: reduced $3$-webs are certain topological types of configurations in the triple-dimer model on a bipartite
surface graph. 
Here we study (reduced) webs \emph{as random objects}. 
We can describe our main result as follows. 
To an edge-weighted planar graph with $n$ specified boundary vertices and specified boundary colors,
we take a random triple-dimer cover. 
The probability that a random reduction of
this triple dimer cover has a specified topological type is then a function of the underlying weights. 
Our main result, Theorem \ref{main}, is an explicit expression for this probability
as a function of the reduced Kasteleyn matrix, or boundary measurement matrix  
(sometimes referred to as \emph{Postnikov's boundary measurement matrix} since Postnikov first made use of it in his study of the 
totally nonnegative Grassmannian \cite{Post}). 
Theorem \ref{main} shows that the probability is a constant times an
integer-coefficient polynomial in the entries of the boundary measurement matrix.

As an illustration of our results, for a planar graph with $3$ white and $3$ black boundary vertices in circular order $w,b,w,b,w,b$,
 the boundary measurent matrix $X$ is a $3\times 3$ matrix
and the probability of the web of Figure \ref{hexwebsmall}
\begin{figure}[htbp]
\begin{center}\includegraphics[width=1.3in]{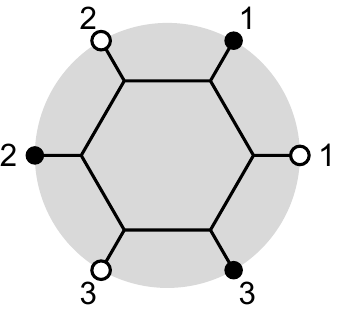}\end{center}
\caption{\label{hexwebsmall}Example of a reduced web with $6$ nodes in order $w,b,w,b,w,b$.}
\end{figure}
 (for ``monochromatic" boundary conditions)
is $Pr = -2\frac{X_{12}X_{23}X_{31}}{\det X}.$
In the scaling limit of the triple dimer model on the upper half plane,
with $6$ marked boundary points as shown in Figure \ref{UHPexample}, and the same monochromatic boundary conditions, 
the probability of this reduced web as a function of the boundary points $z_1<z_2<\dots<z_6\in\R$, is
\be\label{beetle}\text{Pr} = \frac{2(z_2-z_1)(z_3-z_2)(z_4-z_3)(z_5-z_4)(z_6-z_5)(z_6-z_1)}{(z_3-z_1)(z_4-z_2)(z_5-z_3)(z_6-z_4)(z_5-z_1)(z_6-z_2)}.\ee
The five other reduced webs are shown in Figure \ref{Rt61}, and have probabilities with similar expressions, see (\ref{6probs}). See (\ref{PrTn2}) for a case with $6n$ nodes.

\begin{figure}[htbp]
\begin{center}\includegraphics[width=4in]{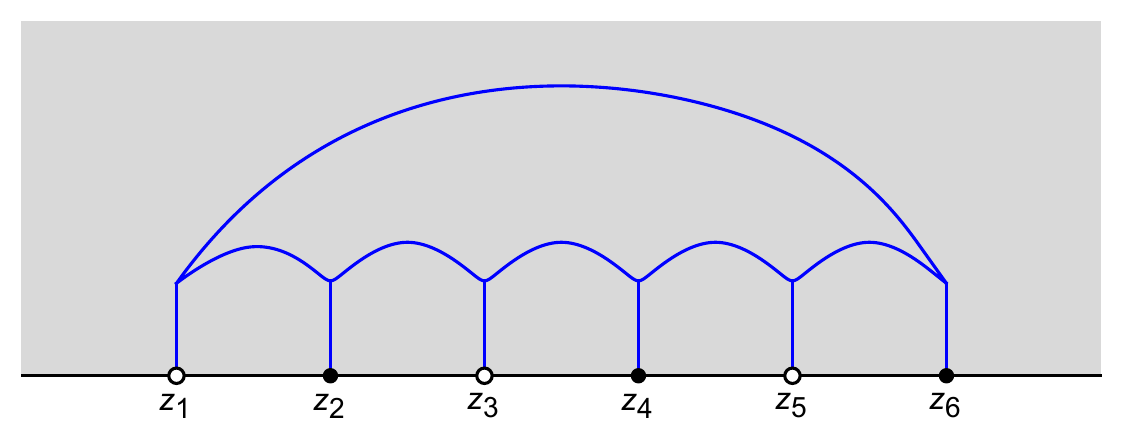}\end{center}
\caption{\label{UHPexample}A reduced web in the upper half plane with $6$ nodes.}
\end{figure}

As application of Theorem \ref{main} we give several infinite families of webs whose probabilities have explicit and compact expressions.
In Section \ref{LGVscn} we give an ``$\text{SL}_3$-version" of the Lindstr\o{}m-Gessel-Viennot theorem,
where we compute the partition function, that is, unnormalized probability, for parallel oriented crossings (see Figure \ref{parallel}) or parallel crossings with ``crossbars" (Figure \ref{crossinggeneral}) as a product of two determinants.
In Section \ref{Tn} we show that the order-$n$ triangular honeycomb web $T_n$ 
of Figure 
\ref{bigweb3} (of which Figure \ref{UHPexample} is the order-$1$ example) has a partition function which is a product of three determinants. The result for another kind of order-$n$ triangular honeycomb (Figure \ref{bigweb4prime}) is given in Section \ref{Tnprime}.

In Section \ref{scalinglimitscn} we discuss the limiting probabilities of reduced webs
for the natural scaling limit of the square grid graph on the upper half plane or other planar domains. 



Here is a short summary of the main result (Theorem \ref{main}) and method of proof. 
Given a circular planar bipartite graph $\G$ define the partition function
$Z:=\sum_{\text{multiwebs $m$}}\Tr_m$. Each $\Tr_m$ can be expanded
in terms of the basis $\{\Tr_\lambda\}_{\lambda\in\Lambda}$ of reduced webs, leading to
$Z = \sum_{\lambda\in\Lambda}C_\lambda \Tr_\lambda,$ for some coefficients $C_\lambda$.
The goal is to compute $C_\lambda$. For each node coloring $c$, we can compute $Z(c)$ as a product of three determinants, 
and expand these using an ``intermediate" spanning set of functions $\{\Tr_\tau\}_{\tau\in\Pi}$ where $\tau$ runs over a simple family of
nonplanar webs. We can then rewrite each $\Tr_\tau$ in terms of the reduced webs $\Tr_\lambda$,
and thus write $\Tr_\tau(c)$ in terms of $\Tr_\lambda(c)$, in such a way that the coefficients $C_\lambda$ do not depend on $c$. Getting the signs right is the hardest part of this theory; the key argument is Lemma \ref{signs}.

\bigskip

\noindent{\bf Acknowledgments.} This research was supported by NSF grant DMS-1940932 and the Simons Foundation grant 327929. 
We thank Daniel Douglas, Pavlo Pylyavskyy and Xin Sun for discussions and insights.

\section{Definitions and background}

\subsection{The graph}\label{graphsection}

Let $\G=(V,E)$ be a circular planar bipartite graph. 
This means $\G$ is a bipartite graph: $V=B\cup W$ with $E\subset B\times W$, 
and $\G$ is embedded in a disk with a specified set of vertices 
$\bV\subset V$ on the bounding circle; $\bV$ contains some but not necessarily all vertices on the outer face of $\G$.  We refer to vertices in $\bV$ as \emph{nodes},
and vertices in $V\setminus\bV$ as \emph{internal vertices}.

We denote by $\bW, \bB$ the white and black nodes,
and by $W_{int},B_{int}$ the white and black internal vertices.
We assume throughout the paper that $|W_{\partial}|\ge |B_{\partial}|$. The case $|W_{\partial}|\le |B_{\partial}|$ is equivalent after exchanging colors. We will also assume 
$$3(|B_{int}|-|W_{int}|)=|W_{\partial}|-|B_{\partial}|,$$
see (\ref{bdybalanced}), below. 

We assume $\G$ is nondegenerate in the following (mild) technical sense. The graph 
$\G$ is required to have two kinds of partial matchings; see Figure \ref{almost} for an illustration.
In the case that $|W_\partial|=|B_\partial|$, we require simply that $\G$ have a partial dimer cover covering exactly the internal vertices, and also that
$\G$ have a full dimer cover covering all vertices. 
More generally, if $|W_\partial|\ge|B_\partial|$, then we require first, that $\G$ have a partial
dimer cover $M$ which covers all internal vertices and no black node 
(and therefore some set of white nodes $W^*$ if $|W_\partial|>|B_\partial|$), see Figure \ref{almost}, left. We let $B_{int}^*\subset B_{int}$ denote the internal black vertices connected in $M$ to $W^*$. Necessarily $|W^*|=|B_{int}^*|=|B_{int}|-|W_{int}|$.
The second required partial dimer cover $M'$ of $\G$ uses all vertices except $W^*\cup W^{**}$ 
where $W^*$ is defined from $M$ as above and $W^{**}$ consists
in, for each $w\in W^*$, one white vertex  
(node or internal) $w'\ne w$ on one of the two external faces adjacent to $w$. (By external face we mean a complementary component of $\G$ in the disk, adjacent to the boundary of the disk.) See Figure \ref{almost}, right. 
The existence of the first kind of partial matching is essential while the existence of the second is probably
not essential but makes our proof easier (as it allows us to reduce the $|\bW|>|\bB|$ case to the $|\bW|=|\bB|$ case....
see Figure \ref{augmented}).
\begin{figure}[htbp]
\begin{center}\includegraphics[width=2.in]{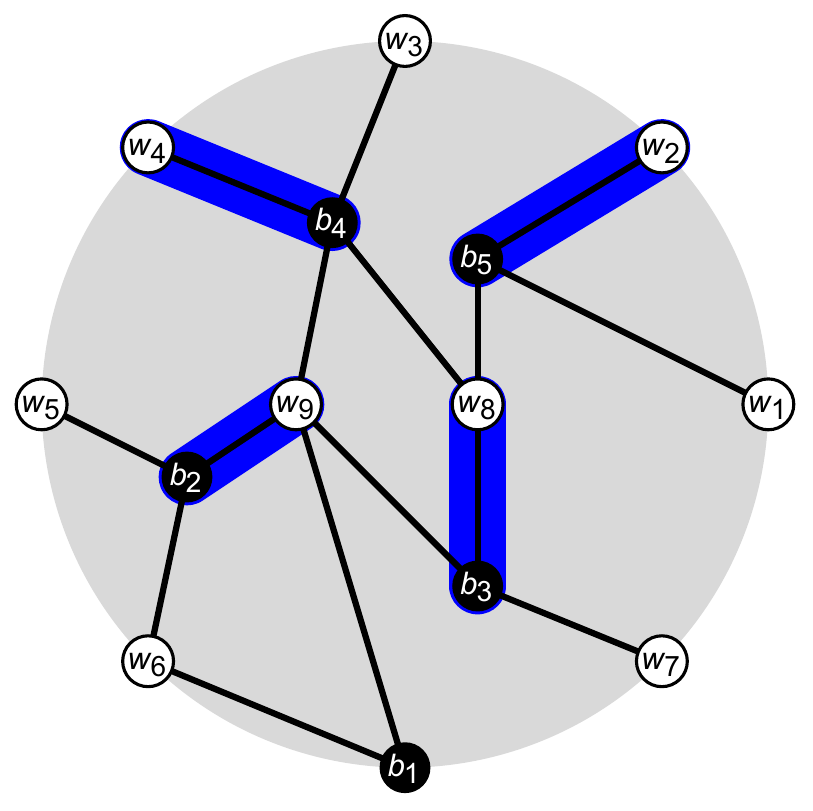}\hskip1cm\includegraphics[width=2.in]{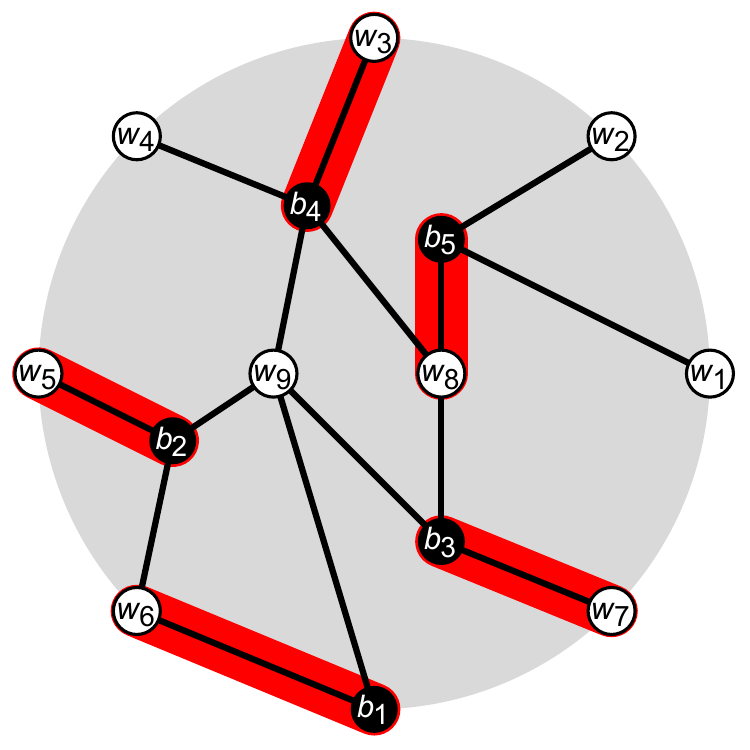}\end{center}
\caption{\label{almost}For nondegeneracy, we need two partial matchings. One, on the left, is a partial matching $M$ using all internal vertices and no black node. 
This has $B_{int}^*=\{b_4,b_5\}$ and $W^*=\{w_2, w_4\}$. 
The second, on the right, is a partial matching $M'$ using all vertices except $W^*=\{w_2, w_4\}$ and $W^{**}=\{w_1,w_9\}$. Note that $w_1$ is on the same external face as $w_2$ and
$w_9$ is on the same external face as $w_4$. }
\end{figure}

We let $\nu:E\to\R_{>0}$ be a positive weight function on the edges of $\G$.

\subsection{Webs, skein relations and reduced web classes}

A \emph{multiweb} in $\G$ is a map $m:E\to\{0,1,2,3\}$ with the property that
for internal vertices $v$, $\sum_{u\sim v} m_{uv}=3$ and for boundary vertices $v$, $\sum_{u\sim v} m_{uv}=1.$
See Figure \ref{blueweb2}. 
In other words a multiweb is a submultigraph with degree $3$ at internal vertices and degree $1$ at each node $v$. 
\begin{figure}[htbp]
\begin{center}\includegraphics[width=2in]{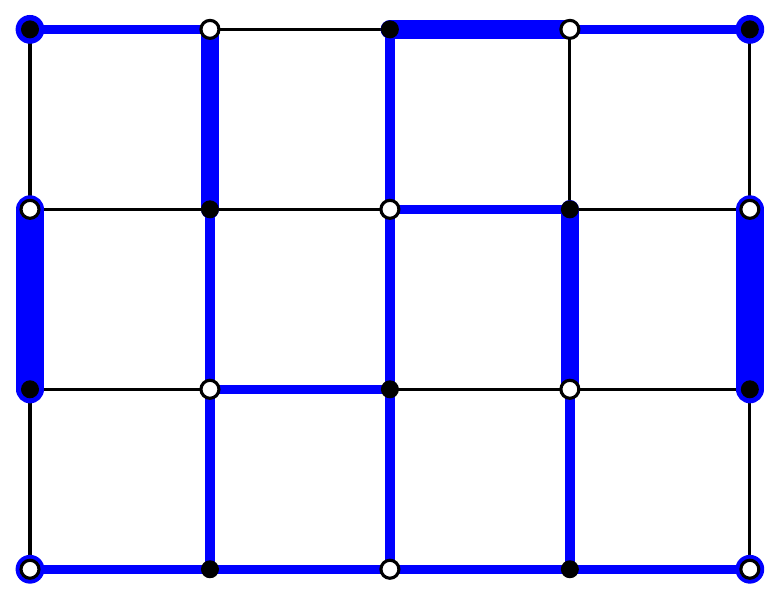}\hskip1cm\includegraphics[width=1.8in]{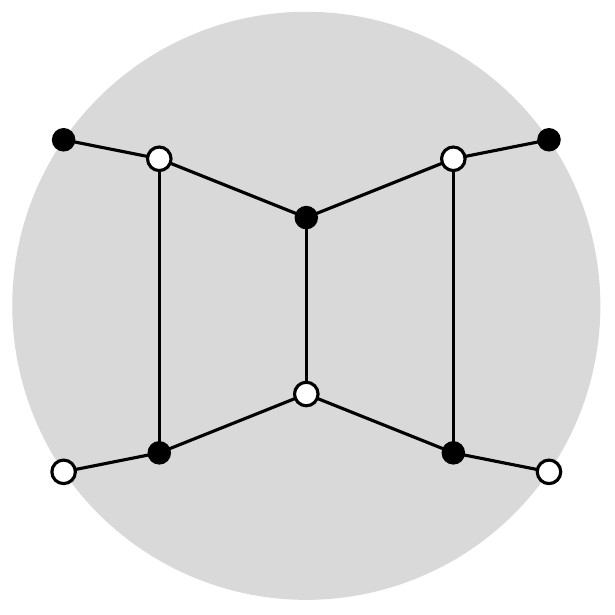}\end{center}
\caption{\label{blueweb2}Left: example of a multiweb (blue) $m$ in a $5\times 4$ grid graph. 
Multiplicities of edges are indicated by their thickness.
Nodes are at the corners (blue dots). Right: the associated abstract web $[m]$.}
\end{figure}
The \emph{weight} $\nu(m)$ of a multiweb $m$ is defined to be the product of its edge weights,
taking into account their multiplicity:
$$\nu(m) = \prod_{e\in E}\nu_e^{m_e}.$$

If we ignore the edges of multiplicity $3$, which are isolated, a multiweb consists of a set of trivalent vertices 
(vertices with three distinct edges of multiplicity $1$)
and the boundary vertices, connected via ``$12$-paths", that is, paths of edges which are alternately of multiplicity $1$ and $2$, beginning and ending with an edge of multiplicity $1$ (these paths may consist in a single edge of multiplicity $1$). 
There may be also some closed loops of $12$-paths.  
A multiweb $m$ has an associated \emph{web} $[m]$, which is an abstract unweighted 
graph with vertices in bijection with the union of the boundary vertices
and trivalent vertices of $m$, and two vertices are connected in $[m]$ if they are connected in $m$ with a $12$-path. In particular the multiplicity-$3$ edges of $m$ are ignored in $[m]$. The abstract web
$[m]$ also has disjoint loops (with no vertices) for each closed $12$-path in $m$.
Note that $[m]$ is a circular planar bipartite graph (possibly with loops) with boundary $\bV$. See Figure \ref{blueweb2}.

A web $[m]$ is \emph{reduced} if it has no faces of degree $0,2$ or $4$.
We say a multiweb $m$ is reduced if its associated web $[m]$ is reduced. Note that in a reduced web, every component contains at least one node:
a component without a node would be a
trivalent bipartite planar graph with no faces of degree $2$ or $4$. A simple Euler 
characteristic argument shows that such graphs do not exist. 
Thus a multiweb is reduced if the connected components of its complement
are bounded by paths of edges which contain either a node or at least $6$ trivalent vertices.

To a multiweb $m$ is associated a function, the \emph{trace} $\Tr_m$, defined in Section \ref{tracedef} below.
It is a multilinear function of a set of vectors in $\R^3$ assigned to the nodes, one to each node.
The trace only depends on the underlying abstract web:  $\Tr_m = \Tr_{[m]}$. 
As shown in \cite{Kuperberg}\footnote{We use the sign convention of Sikora \cite{Sikora} which differs from that of Kuperberg.}, any planar web can be replaced by a formal linear combination of reduced webs by applying a sequence of 
\emph{skein relations}, see Figure \ref{skeinrels}. These skein relations preserve the trace, in the sense that 
the trace of a web equals the sum of the traces of the webs in its reduction. 
\begin{figure}[htbp]
\begin{center}\includegraphics[width=1in]{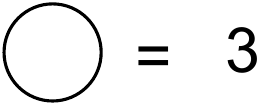}\\\vskip.3cm\includegraphics[width=3in]{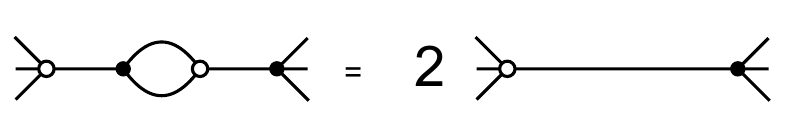}\\\vskip.5cm\includegraphics[width=3in]{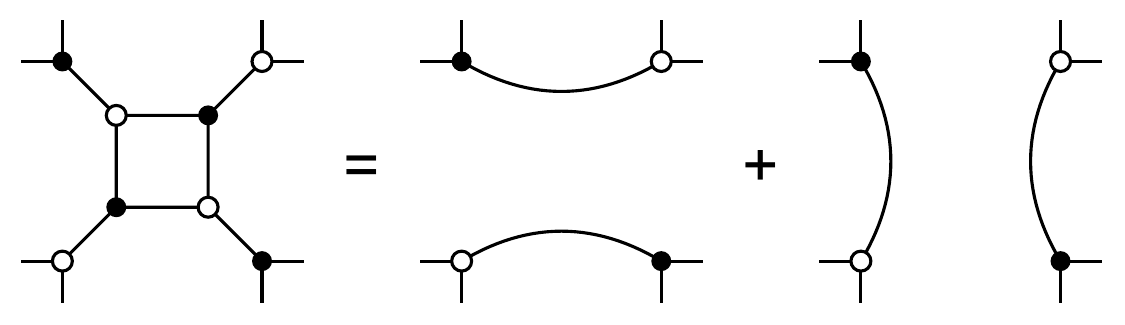}\\\end{center}
\caption{\label{skeinrels}Planar skein reductions for webs. First, any closed loop can be removed, multiplying the coefficient of the remaining web by $3$. Any bigon can be removed as shown, doubling the coefficient. Thirdly, any square face can be replaced by two parallel paths in two ways, as shown.  }
\end{figure}

\begin{figure}[htbp]
\begin{center}\includegraphics[width=3in]{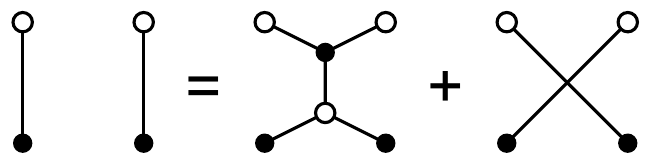}\end{center}
\caption{\label{skeinrel1}The basic skein relation, with Sikora sign convention.}
\end{figure}

The second and third skein relations follow from the more basic relation of Figure \ref{skeinrel1}, which involves
webs with edge crossings. We can use the basic skein relation to replace any nonplanar web (that is, web drawn on the plane with edge crossings) by a formal linear combination of planar webs, by resolving each crossing locally into the other two locally planar pieces. It is worth noting
that if two strands of a web cross more than once consecutively, applying skein relations to 
both these crossings reduces the web to the same web but with both crossings removed,
as in the second Reidemeister move. Likewise if three strands mutually cross then
applying a third Reidemeister move results in an equivalent web (equivalent under skein relations).
Finally, if a strand crosses itself then it is equivalent to the uncrossed version as in the first
Reidemeister move. These observations show that when dealing with webs having edge crossings we are free to isotope the strands across other strands and/or crossings as desired.

The skein relations have analogs for multiwebs as well, see for example Figure \ref{mwskein2}. For a bigon or quad face,
these relations involve changing the multiplicity 
of edges on the face by $\pm1$ alternatively around the face. That is, if multiplicities are $m_1,m_2,\dots,m_{2k}$ then replace these with multiplicities
$m_1+1,m_2-1,\dots,m_{2k}-1$ to get one multiweb and by $m_1-1,m_2+1,\dots,m_{2k}+1$ to get the other multiweb. Both resulting multiwebs have one fewer face.
\begin{figure}[htbp]
\begin{center}\includegraphics[width=4in]{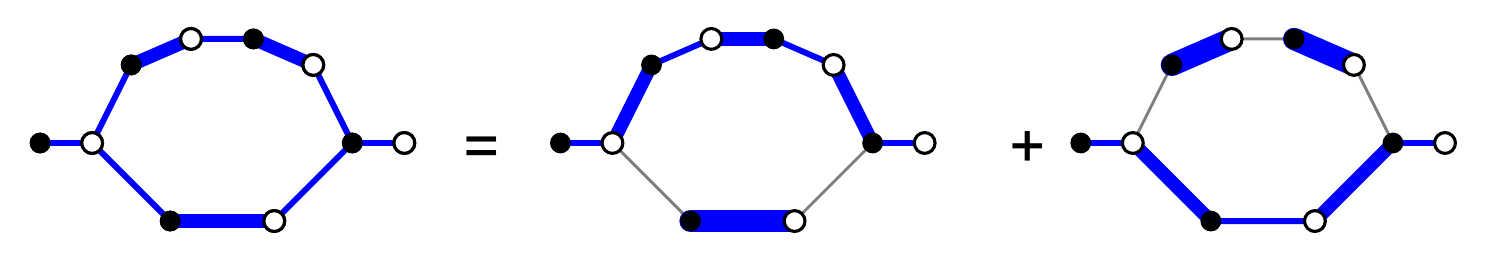}\end{center}
\caption{\label{mwskein2}Example of the second skein relation realized as a multiweb skein relation. Note that it replaces
a multiweb having a bigon face with two different (but equivalent) multiwebs without the bigon face.}
\end{figure}

A \emph{reduced web class} is an equivalence class of reduced multiwebs, two reduced multiwebs being equivalent if
their abstract (reduced) webs are isomorphic with an isomorphism preserving the nodes pointwise. 

Let $\Lambda$ be the set of reduced web classes arising in $\G$.
See Figures \ref{Rt41}, \ref{Rt42}, \ref{Rt51}, \ref{Rt61}, \ref{Rt62}, \ref{Rt63} and \ref{Rt64} for examples.
Given a multiweb $m$ in $\G$, upon applying skein relations we can write
$m$ as a formal integer linear combination of reduced web classes: 
\be\label{mred}m=\sum_{\lambda\in\Lambda} C_{\lambda}(m)\lambda
\ee
where the $C_\lambda(m)$ are nonnegative integers.
See Figure \ref{skeinex} for an example.
\begin{figure}[htbp]
\begin{center}\includegraphics[width=3in]{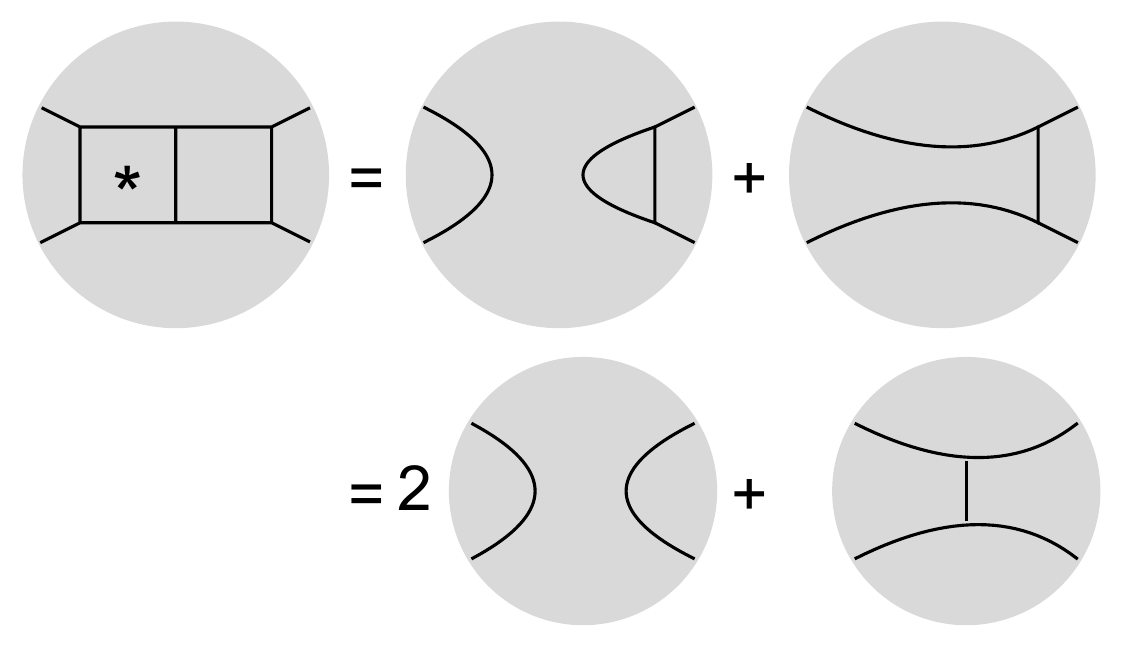}\end{center}
\caption{\label{skeinex}Example of reducing a web with four nodes. We first apply a skein relation to the left quadrilateral face, 
then apply the skein relation to the resulting bigonal face. }
\end{figure}

\begin{figure}[htbp]
\begin{center}\includegraphics[width=4in]{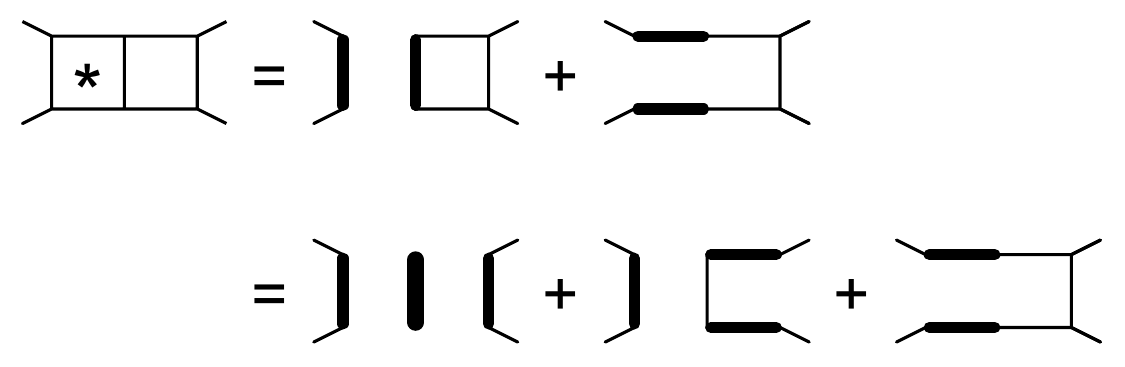}\end{center}
\caption{\label{mwskeinex}The same reduction as in Figure \protect{\ref{skeinex}} using multiweb skein relations. Of the three resulting reduced multiwebs, the first two have the same web class.}
\end{figure}

\old{
We remind the reader that a multiweb $m$ has a \emph{weight}, which is the product (with multiplicity) of its edge weights,
which are positive real numbers.
A multiweb also has a \emph{trace}, which is a positive integer (for given boundary colors), and is the trace of the associated web $[m]$. Skein relations act on the trace but don't change the weight of a web. 
}

\subsection{Node types}

To not confuse the notions of color of a vertex (black or white) with the color of an edge, we refer to the blackness or whiteness of a vertex 
as its \emph{type}.

Fix $\G$ with $n$ nodes $\bV=\{v_1,\dots,v_{n}\}$ in counterclockwise order. We refer to the $n$-tuple $p\in\{b,w\}^n$ of types as the 
\emph{type vector} of $\G$. 
Let $\Omega$ be the set of multiwebs in $\G$, with multiplicities $1$ on $\bV$.
For $\Omega$ to be nonempty there is a linear constraint on the number of black and white internal vertices and
black and white nodes: we must have
\be\label{bdybalanced}3(|B_{int}|-|W_{int}|)=|W_{\partial}|-|B_{\partial}|.\ee
This follows by bipartiteness: $3|W_{int}|+|W_{\partial}|$ is the total white degree, which must equal
the total black degree, since each edge has one black and one white vertex. The quantity $|B_{int}|-|W_{int}|$ is sometimes called the \emph{excedence} in dimer literature.

From (\ref{bdybalanced}), $|\bW|-|\bB|$ must be a multiple of $3$. For example, for
webs with $2,3,4$ or $5$ nodes the possible total types $(|\bW|, |\bB|)$ must be respectively
$(1,1),(3,0),(2,2),(4,1)$, and for $6$ nodes there are two possible total types, $(3,3)$ or $(6,0)$
(recall that we make the standing assumption that $|\bW|\ge |\bB|$).

The maximal number of reduced web classes for $\G$ with $(|\bW|, |\bB|)$ is discussed in Section \ref{3pttnscn} below.

\subsection{Colored multiwebs}

Let $\C=\{1,2,3\}$ be the fixed set of colors. An \emph{edge coloring} of a web $m$ is an assignment of a color in $\C$ to each edge so
that at each trivalent vertex all three colors are present. The definition of edge coloring is slightly different for multiwebs, to take into account the multiplicity of edges. 
A \emph{colored multiweb} is a multiweb in which each edge of multiplicity $k$ has an associated subset $S_e\subset\C$ of size $k$ ($S_e$ 
is the \emph{set of colors of edge $e$}), with the property that at each internal vertex $v$ all colors are present: $\cup_{u\sim v} S_{uv} = \C$. At a node (of degree $1$) only one color will be present. See an example in Figure \ref{webcoloring}.
\begin{figure}[htbp]
\begin{center}\includegraphics[width=1.in]{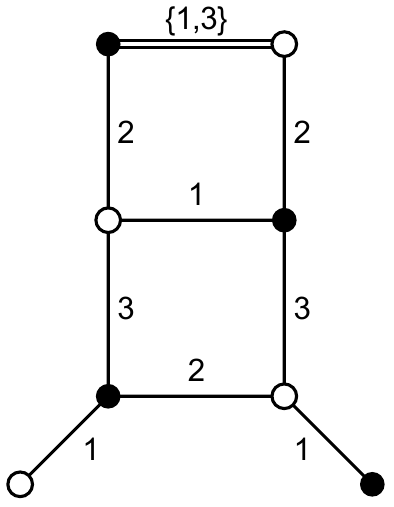}\end{center}
\caption{\label{webcoloring}Coloring of a multiweb, where the lower two vertices are the nodes.}
\end{figure}

Given a choice $c:V_\partial\to\C$ of colors at the nodes, define $\Omega_c$ to be the set of colored multiwebs in $\G$
in which the unique edge at node $v$ has color $c(v)$.  We say $c$ is \emph{feasible} if $\Omega_c$ is nonempty;
$\Omega_c$ is nonempty only when $c$ satisfies certain constraints. 
The set of edges of each color in a colored multiweb is a dimer cover of the graph $\G_i$ obtained from $\G$ by removing
nodes not in $V_i$. So $\Omega_c$ is nonempty if and only if each
$\G_i$ has a dimer cover. 

In particular the number of white nodes of color $i$, plus the number of white internal vertices, must equal the number of black nodes of color $i$ plus the number of black internal vertices:
\be\label{balanced}|W_i|- |B_i|=|B_{int}|-|W_{int}|.\ee
These equations (\ref{balanced}) give a necessary, but not generally sufficient, criterion for $\Omega_c$ to be nonempty.

\subsection{Traces}\label{tracedef}

Let $Y$ be a $3$-dimensional $\R$-vector space, $Y\cong\R^3$, and $Y^*=\text{Hom}(Y,\R)$ its dual. 
Associated to a web $[m]$ with $(|\bW|,|\bB|)=(k,n-k)$ is an $\SL_3(\R)$-invariant multilinear function, its \emph{trace},
$\Tr_{[m]}:Y^{\otimes k}\otimes (Y^*)^{\otimes n-k}\to\R$.
That is, $\Tr_{[m]}$ is a real valued, multilinear function of $k$ vectors in $Y$ and $n-k$ covectors in $Y^*$,
which is invariant under the natural diagonal action of $\SL_3$. For the tensorial definition of $\Tr_{[m]}$ see \cite{Kuperberg} 
or \cite{DKS}. We give here a concrete combinatorial definition, which is equivalent (see \cite{DKS}) 
to the tensorial one, although it has
the disadvantage that the $\SL_3$-invariance is not obvious.

Let $e_1,e_2,e_3\in Y$ be a basis in $Y$ and $e_1^*,e_2^*,e_3^*\in Y^*$ a dual basis.
We associate color $i\in\C$ to $e_i$ and $e_i^*$.
By multilinearity, it suffices to define $\Tr_{[m]}$ on $k$-tuples of basis vectors in $Y$ and $(n-k)$-tuples of basis covectors in $Y^*$.

By definition, the result of applying $\Tr_{[m]}$ to $\vec e = (e_{i_1},\dots,e_{i_{k}},e_{j_1}^*,\dots,e_{j_{n-k}}^*)$
is an integer equal to a signed number of edge colorings of $[m]$ in which the $\ell$th white node has its unique edge of color $i_\ell$
and the $\ell$th black node has its unique edge of color $j_{\ell}$. 
The signs are defined as follows.
Given an edge coloring $\eta$ of $[m]$, 
at each interior vertex the colors are either in
counterclockwise order or clockwise order. At each interior black vertex define the local sign to be $1$ if the order is counterclockwise and $-1$ if clockwise; reverse these conventions at interior white vertices. 
The sign of a coloring is by definition the product over all interior vertices of these local signs, and $\Tr_{[m]}(\vec e)$ 
is the signed number of colored webs whose underlying web is
$[m]$. 

See Figure \ref{detweb} for the most basic examples. The ``tripod" web $[m]$ shown with $(|\bW|,|\bB|)=(3,0)$ has $\Tr_{[m]}(u,v,w) = \det(u,v,w)$,
the determinant of the matrix with columns $u,v,w$. Note that the determinant is multilinear and $\SL_3$-invariant:
$\det(u,v,w)=\det(Au,Av,Aw)$ for any $A\in \SL_3$. Moreover this web has exactly $6$ edge colorings, $3$ of each sign,
which correspond to the six terms
in the expansion (writing $u,v,w$ in the basis $e_1,e_2,e_3$):
$$\det(u,v,w) = u_1v_2w_3 - u_1v_3w_2 + u_2v_3w_1-u_2v_1w_3 + u_3v_1w_2-u_3v_2w_1.$$
The ``line" web $[m]$ shown with $(|\bW|,|\bB|)=(1,1)$ has $\Tr_{[m]}(u,v^*)=v^*(u)$. This is a multilinear function of $u$ and $v^*$,
and also $\SL_3$ invariant (since we \emph{define} the action of $\SL_3$ on $Y^*$ in such a way as to preserve the pairing).
This web has exactly $3$ edge colorings, which correspond to the three terms
in the expansion
$$v^*(u) = v_1u_1+v_2u_2+v_3u_3.$$

\begin{figure}[htbp]
\begin{center}\includegraphics[width=1in]{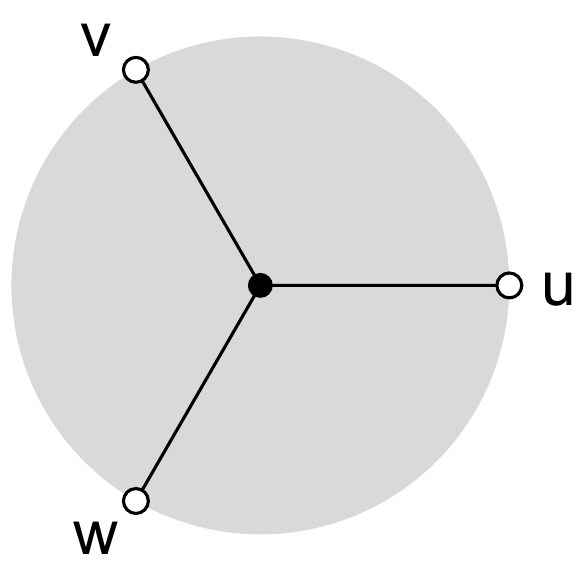}\hskip1in\includegraphics[width=1in]{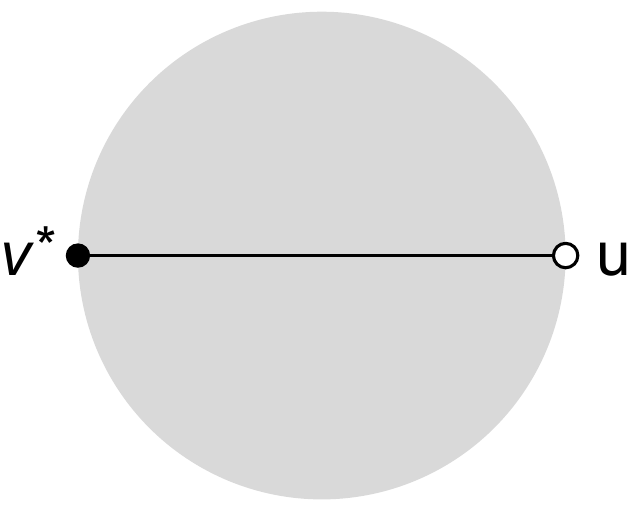}\end{center}
\caption{\label{detweb}Left: the ``tripod" web $[m]$ shown has $\Tr_{[m]}(u,v,w)$ equal to the determinant $u\wedge v\wedge w$. 
Right: the ``line" web $[m]$
has $\Tr_{[m]}(u,v^*) = v^*(u)$.}
\end{figure}

The trace for multiwebs is defined in the same way as the trace for webs, where we only take into
account the trivalent internal vertices when computing the sign. We have $\Tr_{m}=\Tr_{[m]}$.  

In \cite{DKS} it was shown that for a planar graph \emph{without} boundary, all colorings of a web or multiweb have positive sign, so the signed number of colorings equals the number of colorings. A similar result holds for webs with fixed boundary colors.

\begin{lemma}\label{samesign} For a circular planar graph with choice of node colors $c$, 
all colorings of all multiwebs have the same sign $\eps_c$.
\end{lemma}

\begin{proof} We embed $\G$ in a larger planar graph $\G'$ by adding edges outside the bounding disk, in such a way that any colored multiweb $m$ in $\G$ with node colors $c$ extends to a 
colored multiweb $m'$ without boundary in $\G'$. Fix the part of such a 
multiweb $m'$ outside of $\G$.
Then the sign of $m'$, which is positive by \cite{DKS}, 
is the product of the sign of $m$ and the sign of the part of $m'$
outside of $\G$, which is fixed. 
\end{proof} 

One method of computing the sign $\eps_c$ directly is below in Section \ref{3pttnscn}. 

For a reduced multiweb $m$, the trace only depends on the web class $[m]\in\Lambda$.
Skein relations preserve the trace, in the sense that the trace of a web or multiweb
equals the sum of the traces of the webs resulting from it in applying a skein relation.
Applying the trace to both sides of (\ref{mred}) thus gives
\be\label{tracemred}\Tr_{m}=\sum_{\lambda\in\Lambda} C_{\lambda}(m)\Tr_{\lambda}.
\ee
Kuperberg showed in \cite{Kuperberg} that the $\{\Tr_{\lambda}\}_{\lambda\in\Lambda}$ 
are linearly independent on $(Y^{\otimes n_1}\otimes (Y^*)^{\otimes n_2})^{\SL_3}$, that is on the space of $\SL_3$-invariant functions on $Y^{\otimes n_1}\otimes (Y^*)^{\otimes n_2}$.
Thus the coefficients $C_{\lambda}(m)$ in (\ref{tracemred}) are well defined functions of $m$.

\subsection{$3$-partitions}\label{3pttnscn}

Let $\G=(V,E)$ with $|V_\partial|=n$, with $V_\partial = \bW\cup \bB$ as above, $|\bW|\ge|\bB|$, and let $\G$ have type vector $p$. 
Let $\Pi$ be the set of unordered, not-necessarily-planar partitions of the nodes into black/white pairs and white triples.
Let us refer to such partitions as ``$3$-partitions". 
For a $3$-partition, the number of white triples of nodes is exactly $N:=\frac13(W_{\partial}-B_{\partial})=B_{int}-W_{int}$ 
(by (\ref{bdybalanced})).

Letting $k=|\bW|$, the number of elements of $\Pi$ is $\frac{k!}{N!6^N}$. This is because there are $\frac{k!}{(2k-n)!}$ injections from $\bB$ (of size $n-k$) to $W_{\partial}$ (of size $k$),
and then $\frac{(3N)!}{N!(3!)^N}$ ways of partitioning the remaining $3N$ white nodes into triples (and $3N=2k-n$). 

The number of reduced web classes was shown in \cite{PPR} to not depend on $p$, only on
$(|\bW|,|\bB|)$, and to be the Kostka number $K_{3^{m},2^k1^{n-k}}$ where $k=|\bW|$ and $m=(n+k)/3$. 
The Kostka number $K_{\mu,\nu}$ is the number of semistandard Young tableaux (SSYT) of shape $\mu$ and weight $\nu$. In this case, it is the number of ways to fill a $3\times m$ table ($3$ columns, $m$ rows) with numbers $1,1,2,2,\dots,k,k,k+1,k+2,\dots,n$ so that numbers increase along columns and do not decrease along rows.
For example when $(|\bW|,|B_{\partial}|)=(4,1)$, as in Figure \ref{Rt51}, we have $n=5, k=4$ and $m=3$.
The set of SSYT is

\vskip7pt
\centerline{
\begin{tabular}{|r@{$\,\,\,$}|c@{$\,\,\,$}|c@{$\,\,\,$}|c@{$\,\,\,$}}
        \hline
4&4&5 \\\hline
2&3&3\\\hline
1&1&2\\\hline
\end{tabular}\hskip1cm
\begin{tabular}{|r@{$\,\,\,$}|c@{$\,\,\,$}|c@{$\,\,\,$}|c@{$\,\,\,$}}
        \hline
3&4&5 \\\hline
2&3&4\\\hline
1&1&2\\\hline
\end{tabular}\hskip1cm
\begin{tabular}{|r@{$\,\,\,$}|c@{$\,\,\,$}|c@{$\,\,\,$}|c@{$\,\,\,$}}
        \hline
3&4&5 \\\hline
2&2&4\\\hline
1&1&3\\\hline
\end{tabular}}
\vskip7pt
\noindent and the Kosta number is $K_{3^3, 2^41^1} = 3$.

We can associate to a $3$-partition a crossing diagram in the disk, by drawing a chord connecting pairs of the $3$-partition and, for every triple $t=w_1w_2w_3$, drawing a ``tripod":
putting an interior black vertex $b=b_t$ in the disk and connecting it to each of $w_1,w_2,w_3$ with a line segment.

The trace of a $3$-partition $\tau$ 
is defined using the above drawing as for the trace of a web; the nonplanarity plays no role in the definition of trace:
if $\tau\in\Pi$ and $c$ is a compatible coloring of the nodes, then $\Tr_\tau(c)=(-1)^\ell$ where $\ell$ is the number of 
triples of $\tau$ whose colors are in clockwise order. In particular if there are no triples, $\Tr_\tau(c)=1$. 

One way to compute the sign of $\eps_c$ from Lemma \ref{samesign} is as follows. Find a $3$-partition $\tau$ of the nodes of $\G$. Draw $\tau$ in the disk (with possibly crossing edges) by joining each triple of $\tau$ 
with a black vertex $b$ in the disk, and drawing each pair of $\tau$ as a path connecting its endpoints.
Each path in $\tau$ inherits a color from its boundary vertex/vertices. 
Let $T$ be the number of triples with colors in clockwise order and
let $M$ be the number of nonmonochromatic crossings (NMCs): crossings of edges of different colors.  
Note that while the number of NMCs depends on the choice of drawing, its parity is an isotopy invariant.
Then $\eps_c=(-1)^{M+T}$. This follows because, using the basic skein relation (see Figure \ref{skeinrel1}) to make each crossing planar, 
each NMC resolves into a double-Y with one vertex of each sign, and each monochromatic crossing 
resolves into a pair of parallel edges (with no sign change). We have constructed a colored planar web with sign $\eps_c$,
and thus by Lemma \ref{samesign} all planar webs with boundary colors $c$ have this same sign $\eps_c$.

\subsection{Kasteleyn matrix and boundary measurement matrix}
\label{Kastsection}

Associated to $\G$ and $\nu$ is a \emph{Kasteleyn matrix} $K$, a matrix with rows indexing the white vertices and columns indexing the black vertices. This is a signed, weighted adjacency matrix: $K_{wb}=0$ if $w$ is not adjacent to $b$ and, if they are adjacent, $K_{wb} = \eps_{wb}\nu(wb)$
where signs $\eps_{wb}\in\{1,-1\}$ are chosen so that a bounded face of length $\ell$ has $\frac{\ell}2+1\bmod 2$ minus signs.  See e.g. \cite{Kast, Kenyon.lectures}. 
Different choices of such signs give to matrices related by ``gauge transformation", that is, by left- and right-multiplication by diagonal matrices of diagonal entries $\{\pm1\}$.

We define $X$, the \emph{reduced Kasteleyn matrix}, or \emph{boundary measurement matrix}, 
to be a maximal Schur reduction of $K$ to the nodes,
in the following sense.
Take a partial matching $M$ which uses all internal vertices and no black node. Such a matching exists by nondegeneracy, see Figure \ref{almost}, left.
Let $B_{int}^*\subset B_{int}$ denote the internal black vertices connected in $M$ to a white node.
Let $D$ be the matrix $K$ restricted to the vertices $W_{int}$ and $B_{int}\setminus B_{int}^*$: this is the set of vertices of $M$ when we remove from $M$ the dimers connected to white nodes.
Then $D$ is a Kasteleyn matrix of its induced subgraph: 
each bounded face of the subgraph is a face of $\G$.
By ordering vertices of $\G$ so that $D$'s vertices are last, we can write 
$K=\begin{pmatrix} A&B\\C&D\end{pmatrix}$ in block form.
Then define 
\be\label{Xdef}X:=A-BD^{-1}C.\ee
Let $\Delta=\det D$; this is nonzero since $D$ is a Kasteleyn matrix of a subgraph of $\G$ which
has at least one dimer cover. 

Up to sign each entry $X_{wb}$ has the following combinatorial interpretation \cite{KW11}:
$X_{wb}= \pm Z(\G_{wb})/\Delta$ where $Z(\G_{wb})$ is the weighted sum of dimer covers of the graph $\G_{wb}$, the graph obtained
from $\G$ by removing all nodes and all of $B_{int^*}$ except $w$ and $b$. 
To see this, note that
removing rows and columns of $K$ for all nodes and vertices $B_{int}^*$ not equal to $w$ or $b$ results in a matrix 
$\begin{pmatrix}A_{wb}&B_w\\C_b&D\end{pmatrix}$ whose Schur reduction is 
the $1\times1$ matrix $X_{wb}$, and whose determinant is $Z(\G_{wb})$.

Note that $X$ depends nontrivially on
our choice of $B_{int}^*$ (corresponding to submatrix $D$).

If $|\bW|=|\bB|$ then $K$ and $X$ are square matrices, and 
\be\label{detX}\det K = \det D\det X=\Delta\det X;\ee 
this follows from applying the determinant to both sides of
$$\begin{pmatrix}A&B\\C&D\end{pmatrix} = \begin{pmatrix}A-BD^{-1}C&B\\0&D\end{pmatrix}\begin{pmatrix}I&0\\D^{-1}C&I\end{pmatrix}.$$
Note that $\det K\ne 0$ by nondegeneracy.

\subsection{Signs}

If $|\bW|=|\bB|=k$, then $K$ is a square $k\times k$ matrix. 
Suppose $W_1\cup W_2\cup W_3$ is a partition of the white nodes
and $B_1\cup B_2\cup B_3$ is a partition of the black nodes, with $|W_i|=|B_i|$ for each $i$.
Let $c$ denote this partition; we think of $c$ as a coloring of the nodes with colors $\C=\{1,2,3\}$.
Let $K_i =K_{W_i}^{B_i}$ be the submatrix of $K$
obtained by discarding nodes not in $W_i\cup B_i$, that is, keeping only interior vertices and nodes of color $i$.
Likewise define $X_i = X_{W_i}^{B_i}$, its Schur reduction using the same submatrix $D$. We have $\Delta \det X_i=\det K_i$ (and $\Delta \det X = \det K$, see above).

Let $\delta_c$ be the sign of the quantity $\det K_1\det K_2\det K_3/(\det K)^3$, which is also the sign of $\det X_1\det X_2\det X_3/(\det X)^3$ and the sign of $\det X_1\det X_2\det X_3/\det X$.
We define another sign, $\eta_c$, which is the sign of $\det X_1\det X_2\det X_3$
as a polynomial in the entries of $X$ occurring in the expansion of $\det X$.

It is convenient to choose an ordering for the nodes, which determines an ordering
for the rows and columns of $X$.
Index the nodes of $\G$
counterclockwise for white vertices and clockwise for black vertices.
Letting $w_1$ be the first white node, we assume that $b_1$ is the first black node clockwise from it. This particular ordering is just a convention, and helps with the proof of Theorem \ref{main} (although it does not affect the statement). We take a different order in Section \ref{Tn} and the examples in the appendix.

\begin{lemma}\label{signs} We have $\eps_c=\delta_c\eta_c$.
\end{lemma}

\begin{proof}
Consider all the vertices on the outer face of $\G$; some subset of these are nodes. 
As noted above we index the nodes
counterclockwise for white vertices and clockwise for black vertices.
Let $w_1$ be the first white node, and assume that $b_1$ is the first black node clockwise from it.

Let us compute the sign of $\det X_1/\det X= \det K_1/\det K$. We add external edges to $\G$ outside the disk, 
connecting white nodes $W_2\cup W_3$ to black nodes $B_2\cup B_3$ in pairs, and so that the resulting graph $\G^1$ is still planar. 
We can extend the Kasteleyn signs on $\G$ to $\G^1$ by assigning signs to the new edges. 
Let $K^1$ be the enlarged Kasteleyn matrix. In $\det K^1$, all dimer covers have the same sign by Kasteleyn's theorem; 
but dimer covers which use none of the new edges have the
same sign as in $\det K$. This sign is also the sign of dimer covers using all the new edges, which by definition have the sign of $\det K_1$ times the product of the new signs.
Thus $\det K_1/\det K$ has sign given by the product of the Kasteleyn signs on the new edges.
Likewise construct $\G^2,\G^3$ and $\det K_2/\det K, \det K_3/\det K.$

By the gauge invariance of the choice of Kasteleyn signs, we can choose Kasteleyn signs on the original graph $\G$ so that all boundary edges have sign $+$,
except possibly for the edge $e_0$ 
just clockwise of $w_1$ which has sign $-1$ if the boundary has length $0\bmod 4$.
Then the Kasteleyn sign on an external edge of $\G^i$ depends on the distance $\ell$ in $\G$ between its vertices, measured
around the arc of the boundary which does not cross $e_0$: the sign is $(-1)^{(\ell-1)/2}$.

Thus the sign of each $\det X_i/\det X$ depends only on the distances (mod $4$) between nodes.

Suppose we move a black node $b_i$, of color $1$, increasing its distance from the previous node by $2$ and decreasing its distance from the next node by $2$. That is, we make $b_i$ an internal
vertex and instead assign the next (clockwise) black vertex around the boundary to be a node,
the ``new" $b_i$.
Then both $\det X_2/\det X$ and $\det X_3/\det X$ will change sign (since whatever white node $b_i$ was matched to externally will
have its distance to $b_i$ changed by $\pm2$), while $\det X_1/\det X$ will not change sign.  This results in no net sign change of $\delta_c$. Moreover $\eps_c,\eta_c$ do not change under this operation. 
So we see that $\delta_c,\eps_c$ and $\eta_c$ 
only depend on the circular order of the nodes (and their colors), not their individual distances. 

Suppose initially that white nodes and black nodes lie in disjoint intervals, so that counterclockwise
starting from $w_1$ we see $w_1,\dots,w_n, b_n,\dots,b_1$. Suppose also that colors are in the 
counterclockwise order $W_1,W_2,W_3$ then $B_3,B_2,B_1$. That is, 
starting from the first white node $w_1$ and proceeding counterclockwise we see all white nodes
of color $1$, then all white nodes of color $2$, and so on, ending with all black nodes of color $1$.
In this case we can add external edges in each of $\G^1,\G^2,\G^3$ connecting each $w_i$ to $b_i$. In $\det X_1\det X_2\det X_3/(\det X)^3$ each sign on these
external edges contributes twice, so $\delta_c=1$. Moreover $\eps_c=1$ and $\eta_c=1$ in this case, so the statement holds.

It remains to see how the sign changes when we reorder the boundary nodes or colors. There are three cases to consider:
we switch colors on adjacent nodes, both of which are black or both of which are white; we switch positions of 
a black and white adjacent node of the same color;
and we switch positions of a black and white adjacent node of different colors. In each case we compute the change in sign of $\delta_c,\eta_c$ and $\eps_c$ and show that $\eps_c\delta_c\eta_c$ remains constant.

Suppose we switch colors, say color $1$ and $2$, on two adjacent black nodes, separated by (even) 
distance $\ell$. Then both $\eps_c$ and $\eta_c$ will change.
In both $\det X_1/\det X$ and $\det X_2/\det X$ the sign will change by $(-1)^{\ell/2}$, while the sign of 
$\det X_3/\det X$ will not change. Thus $\delta_c$ changes by $(-1)^\ell=+1$, that is, does not change. 
A similar argument holds if we switch colors on two adjacent white nodes.

Suppose we switch positions of two adjacent nodes $b,w$, one black and one white, of the same color, say color $1$.
Then $\eta_c$ and $\eps_c$ do not change. 
In the graphs $\G^2$ and $\G^3$, we can assume $b,w$ are paired by an external edge, since they are adjacent. 
When we swap them, keeping the same distance, the
signs of $\det X_2/\det X$ and $\det X_3/\det X$ do not change; neither does the sign of $\det X_1/\det X$. Thus $\delta_c$ does not change.  

Suppose we switch positions of two adjacent nodes $b,w$, one black and one white, of odd distance $\ell$, and of different colors, say colors $1$ and $2$ without loss of generality (and we think of colors as attached to the nodes, so the colors also swap positions). Suppose for example $b$ advances by $\ell+1$ cclw around the boundary, and $w$ retreats by $\ell-1$. Then $\eps_c$ changes, and $\eta_c$ does not change.
In the graph $\G^3$, there is no sign change (we can assume $w,b$ are paired in $\G^3$). In the graph $\G^1$ or $\G^2$,
there is a sign change of $(-1)^{(\ell-1)/2}$ and $(-1)^{(\ell+1)/2}$ respectively, resulting in $(-1)^\ell = -1$, a net sign change.
Consequently $\eps_c\delta_c\eta_c$ is invariant in this case as well.
 \end{proof}

\subsection{Partition function}

Fix $\G$ as before with $n$ nodes, satisfying (\ref{bdybalanced}), $k$ of which are white and $n-k$ of which are black.
Let $\Omega$ be the set of multiwebs in $\G$.

We define the \emph{partition function} $Z=Z(\G)$ by
\be\label{Z}Z=\sum_{m\in\Omega}\nu(m)\Tr_{m} =\sum_{m\in\Omega}\nu(m)\sum_{\lambda\in\Lambda}C_\lambda(m)\Tr_{\lambda}\ee
where the $C_\lambda(m)$ are defined in (\ref{mred}).
Then $Z$ is itself an $\SL_3$-invariant multilinear function on $Y^{\otimes k}\otimes (Y^*)^{\otimes n-k}$.
By interchanging the order of summation we can write
\be\label{Zdef}Z=\sum_{\lambda\in\Lambda}C_{\lambda}\Tr_{\lambda}\ee
where $C_{\lambda} = \sum_{m\in\Omega}\nu(m)C_{\lambda}(m).$
We call $C_\lambda$ the \emph{partition function for class $\lambda$}.

Our main result is a computation of $C_\lambda$.
We show that 
\be\label{Pdef} C_\lambda = \Delta^3 P_\lambda
\ee where $P_\lambda$ is a certain integer-coefficient polynomial in the entries in the boundary measurement matrix $X$,
and $\Delta$, which is independent of $\lambda$, 
is the determinant of a submatrix of the associated Kasteleyn matrix, as defined in Section \ref{Kastsection} above.

If $c:V_\partial\to\C$ be a feasible coloring of the nodes, 
then $Z(c)$ is a number, and 
\be\label{Zcdef}Z(c)=\sum_{m\in\Omega}\nu(m)\Tr_{m}(c) =\sum_{\lambda\in\Lambda}C_{\lambda}\Tr_{\lambda}(c).\ee

By Lemma \ref{samesign}, for feasible $c$, all terms in (\ref{Zcdef}) have the same sign, and so $Z(c)$ defines a natural probability measure $\text{Pr}_c$ on $\Lambda$, with
\be\label{Pr}\text{Pr}_c(\lambda) = \frac{C_\lambda \Tr_{\lambda}(c)}{Z(c)}.\ee
This is the probability that a ``random reduction" of a random multiweb (with boundary colors $c$) will have shape $\lambda$. More precisely, choose a ($\Tr_c$-weighted) random multiweb $m$ with boundary colors $c$, and consider all possible reduced webs (with multiplicity) after applying skein relations to $[m]$.
Choose one of these uniformly at random, and consider its shape $\lambda$.

We have an explicit determinantal expression for $Z(c)$.
Let $V_\partial = V_1\cup V_2\cup V_3$ be the partition of $V_\partial$ into the nodes of color $1,2,3$ respectively.
Let $\G_i$ be the graph obtained from $\G$ by removing nodes not in $V_i$. Note that $\G_i$ is balanced:
it has the same number of black vertices as white vertices, by (\ref{balanced}).
Let $Z_i = |\det K(\G_i)|$, the determinant of the Kasteleyn matrix of $\G_i$. Here we can take
$K(\G_i)$ to be the Kasteleyn matrix
of $\G$ restricted to $\G_i$, since removing vertices from the outer face of $\G$ yields a submatrix of $K$ also satisfying the Kasteleyn condition. 

\begin{lemma}\label{Zlem} We have
$Z(c) = \eps_cZ_1Z_2Z_3$ where $\eps_c$ is the sign from Lemma \ref{samesign}.
\end{lemma}

\begin{proof} 
By (\ref{Zcdef}), $Z(c)$ is weighted sum of multiwebs, multiplied by their traces. By Lemma \ref{samesign}, 
the trace is a global sign $\eps_c$ 
times the number of colorings. So $Z(c)$ is $\eps_c$ times the weighted number of colored multiwebs with node colors $c$.
A colored multiweb with node colors $c$ is the same as a $3$-tuple of dimer covers, one of each color, where the dimer cover of color $i$ is a dimer cover
of the graph $\G_i$. 
But $Z_i=|\det K(\G_i)|$ is the positive weighted sum of dimer covers of $\G_i$.
\end{proof}

By definition of $\delta_c$ we can write
$Z_1Z_2Z_3=\eps_K\delta_c\det K(\G_1)\det K(\G_2)\det K(\G_3)$ 
where $\eps_K$ is the sign of $\det K$.
Thus
\be\label{eee}Z(c) = \eps_c\delta_c\eps_K\det K(\G_1)\det K(\G_2)\det K(\G_3).\ee

\section{Examples}

We illustrate here the computation of partition functions and probabilities for reduced webs in the cases when 
$\G$ has $4$ and $5$ nodes. 
Cases with $2$ or $3$ nodes are trivial, since there is only one reduced web class in these cases, see Figure \ref{detweb}.
In the appendix we illustrate the cases with $6$ nodes.

\subsection{$(w,b,w,b)$}

When $n=4$, by (\ref{bdybalanced}) we must have $|\bW|=|\bB|=2$, and there are two possible types, up to rotational symmetry, which are $p=(w,b,w,b)$ and $p=(w,w,b,b)$.

Suppose $p=(w,b,w,b)$.
The matrix $X$ is $2\times 2$. 
The set $\Lambda$ consists of two classes of reduced webs, $\lambda_1,\lambda_2$, shown in Figure \ref{Rt41}.
\begin{figure}[htbp]
\begin{center}\includegraphics[width=2.5in]{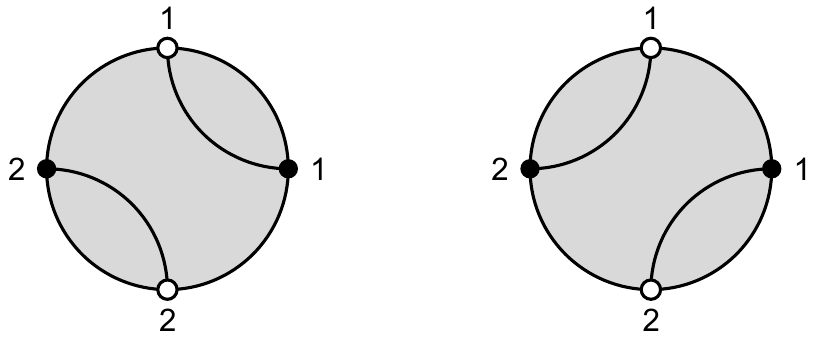}\end{center}
\caption{\label{Rt41}The two classes of webs with $4$ boundary points and $p=(w,b,w,b)$.}
\end{figure}
As a consequence of Theorem \ref{main} below, we have (up to a global sign)
\be\label{4n1}P_{\lambda_1}= X_{11|22} = X_{1,1}X_{2,2}\ee
and
\be\label{4n2}P_{\lambda_2} = X_{12|21} = -X_{1,2}X_{2,1}\ee
where $P_\lambda$ is defined in (\ref{Pdef}).

For example take color $c\equiv 1$ for the four nodes. The traces of both webs are $1$.
The probabilities of the two webs are then
\be\label{4ptprobs}Pr_c(\lambda_1) = \frac{X_{1,1}X_{2,2}}{X_{1,1}X_{2,2}-X_{1,2}X_{2,1}},~~~Pr_c(\lambda_2) = \frac{-X_{1,2}X_{2,1}}{X_{1,1}X_{2,2}-X_{1,2}X_{2,1}}.\ee

For a concrete setting for this example, consider the graph $\G$ of Figure \ref{4nodeexample}, with edge weights indicated.
\begin{figure}[htbp]
\begin{center}\includegraphics[width=1.7in]{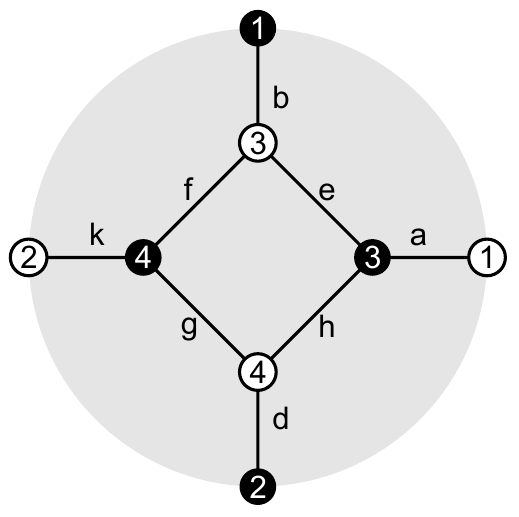}\end{center}
\caption{\label{4nodeexample}Graph with four nodes and $p=(w,b,w,b)$.}
\end{figure}
Then $\Omega_c(\G)$ has $3$ multiwebs, two of which are already reduced and the third of which is reducible into one of each web class.
We have
\be\label{zc1}Z(c) =  abde^2g^2k + abdf^2h^2k + 2abdefghk\ee
which is a sum of the weights (times traces) of the three multiwebs.
The last term corresponds to a non-reduced web with a quad face, which can be reduced to one copy of each reduced web class. 
We thus have 
\begin{align*}C_{\lambda_1} &= abde^2g^2k+abdefghk\\
C_{\lambda_2} &= abdf^2h^2k+abdefghk.
\end{align*}

A Kasteleyn matrix for the given vertex order is
$$K=\begin{pmatrix}
0&0&b&0\\0&0&0&d\\
a&0&e&h\\0&k&-f&g
\end{pmatrix}
$$
for which we find
$$Z(c) = Z(\G)Z(\G_{int})^2= abdk(eg+fh)^2,$$ 
equal to (\ref{zc1}),
and the corresponding reduced Kasteleyn matrix is
$$X= \frac{1}{\Delta}\begin{pmatrix}-abg&bkh\\-adf&-kde\end{pmatrix}$$
where $\Delta = \det D = \det\begin{pmatrix}e&f\\-h&g\end{pmatrix} = eg+fh.$
From this one can verify (\ref{4n1}) and (\ref{4n2}). Note that a different choice of Kasteleyn signs might change the sign of (both of) $P_{\lambda_1},P_{\lambda_2}$ but the probabilities remain unchanged.

Suppose we had chosen a different set of node colors $c'$, for example color $1$ at black and white nodes $1$ and color $2$ at black and white nodes $2$.
Then $\lambda_2$ is incompatible with $c'$, and so $\Tr_{\lambda_2}(c)=0$. We still have $\Tr_{\lambda_1}(c)=1$.
In this case 
\begin{align*}
Z_{c'} &= Z(\G\setminus\{2,{\bf 2}\})Z(\G\setminus\{1,{\bf 1}\})Z(\G_{int})\\
& = (abg)(dek)(eg+fh)\\
&=abde^2g^2k+abdefghk\\ 
&= C_{\lambda_1}\\
&=C_{\lambda_1}\Tr_{\lambda_1}(c') + C_{\lambda_2}\Tr_{\lambda_2}(c').
\end{align*}

\subsection{$(w,w,b,b)$}

The other case with $n=4$ has $p=(w,w,b,b)$. 
Again $\Lambda$ has two different reduced web classes, shown in Figure \ref{Rt42}.
\begin{figure}[htbp]
\begin{center}\includegraphics[width=3in]{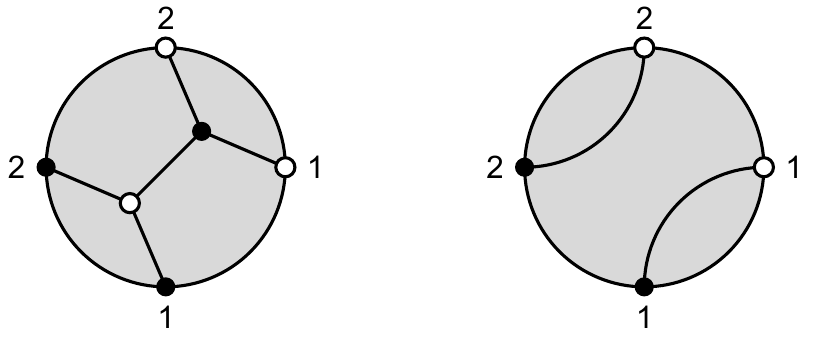}\end{center}
\caption{\label{Rt42}The two classes of webs with $4$ boundary points and $p=(w,w,b,b)$.}
\end{figure}

Again $X$ is a $2\times2$ matrix. Theorem \ref{main} shows that, up to a global sign,
$$P_{\lambda_1}= X_{1,2}X_{2,1}$$
and
$$P_{\lambda_2} = X_{1,1}X_{2,2}-X_{1,2}X_{2,1}.$$

As a concrete realization,
let $\G$ be the $3\times 2$ grid graph of Figure \ref{2by3}.
A Kasteleyn matrix is 
$$K=\begin{pmatrix}a&0&b\\0&k&d\\-g&-e&f
\end{pmatrix}$$ and 
$$X=\begin{pmatrix}a+\frac{bg}f&\frac{be}f\\\frac{dg}f&k+\frac{de}f\end{pmatrix}$$
with $\Delta = f$.  

\begin{figure}[htbp]
\begin{center}\includegraphics[width=1.5in]{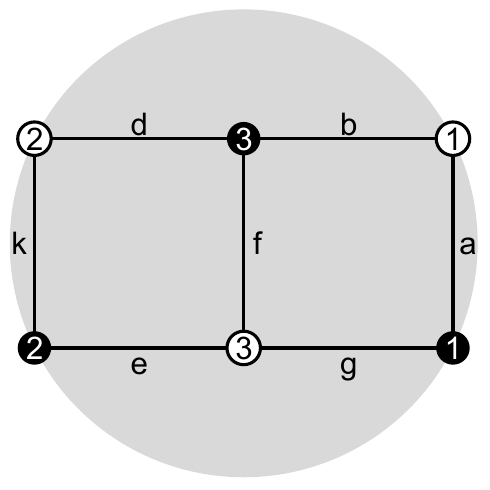}\end{center}
\caption{\label{2by3}Graph with four nodes and $p=(w,w,b,b)$.}
\end{figure}

\subsection{$(b,w,w,w,w)$}

When $n=5$ there is, up to rotation, only one possible type pattern, $p=(b,w,w,w,w)$. 
In this case $\Lambda$ consists in 3 classes $\lambda_1,\lambda_2,
\lambda_3$, shown in Figure \ref{Rt51}. Now $X$ is a $4\times 2$ matrix,
and Theorem \ref{main} shows that, up to a global sign,
\begin{align} \nonumber P_{\lambda_1}&=X_{1,2}X_{2,2}(X_{3,1}X_{4,2}-X_{4,1}X_{3,2})\\
\label{Ppwwww}P_{\lambda_2}&=X_{3,2}X_{4,2}(X_{1,1}X_{2,2}-X_{2,1}X_{1,2})\\
\nonumber P_{\lambda_3}&=X_{1,2}X_{4,2}(X_{2,1}X_{3,2}-X_{3,1}X_{2,2}).
\end{align}

\begin{figure}[htbp]
\begin{center}\includegraphics[width=4in]{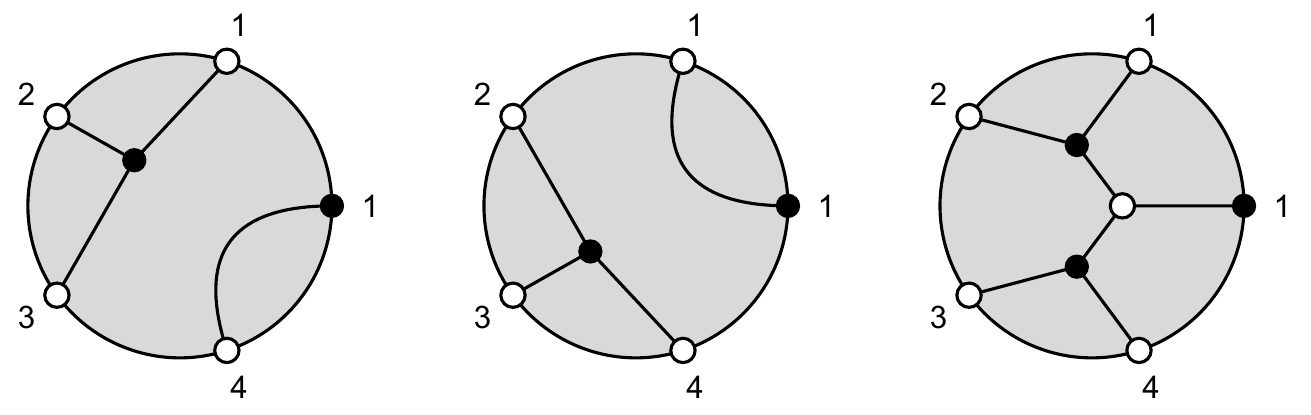}\end{center}
\caption{\label{Rt51}Classes of webs with $5$ boundary points and $p=(b,w,w,w,w)$.}
\end{figure}

\section{Reduction matrix}\label{sgnscn}

\subsection{The reduction matrix $\P$}

For a fixed type vector $p$, we define a matrix $\P=\P(p)$ with rows indexed by reduced web classes $\Lambda$ and columns indexed by $3$-partitions $\Pi$, as follows.
Using the basic skein relation (\ref{skeinrel1}) each $3$-partition $\tau\in \Pi$ can be converted into a unique formal 
linear combination of planar web classes $\lambda\in \Lambda$:
\be\label{tausum}\tau=\sum_{\lambda\in \Lambda}\P_{\lambda,\tau}\lambda.\ee
That is, 
each $3$-partition can be represented as a nonplanar diagram in a disk, as in Section
\ref{3pttnscn}. Now 
replace each crossing
with the corresponding linear combination of the two locally planar resolutions. Reduce each resulting planar diagram
using the planar skein relations (\ref{skeinrels}). The resulting
linear combination of reduced planar webs is unique, independent of the original drawing of the $3$-partition in the disk (only depending on its isotopy class) and independent of the order of
reductions.
We call the matrix $\P$ the \emph{reduction matrix from $3$-partitions to reduced web classes},
or \emph{reduction matrix} for short.

As an example, consider the case $n=5$ with $p=(b,w,w,w,w)$. Ordering $\Lambda$ as in Figure \ref{Rt51},
and using $\Pi = \{{\bf1}1|234, {\bf1}2|134, {\bf1}3|124,{\bf1}4|123\}$ (where boldface denotes black nodes), the reduction matrix $\P$ is
given in table \ref{pwwwwtable}.

\begin{table}
\centerline{
\begin{tabular}{r@{$\,\,\,$}|c@{$\,\,\,$}c@{$\,\,\,$}c@{$\,\,\,$}c@{$\,\,\,$}}
        &\rf{{\bf1}1|234} & \rf{{\bf1}2|134} & \rf{{\bf1}3|124}  & \rf{{\bf1}4|123} \\
        \hline
$1$  &0&0&1&1 \\
$2$   &1&1&0&0\\
$3$   &0&-1&-1&0
\end{tabular}}
\caption{\label{pwwwwtable} The reduction matrix for $p=(b,w,w,w,w)$.}
\end{table}


\subsection{$X_\tau$ variables}

For $\tau\in \Pi$ define variables $X_\tau$ as follows.
First, if there are no triples, that is, if $|\bW|=|\bB|$, then we can think of $\tau$ as a bijection $\sigma$ from white nodes to black nodes, that is, from $[n]$ to $[n]$ (identifying nodes with rows and columns of $X$)
and in this case 
\be\label{Xtaupair}X_\tau=(-1)^\sigma\prod_{w\in\bW}X_{w\sigma(w)}.\ee
This is consistent with the definition in \cite{KW11}. 

If there are triples in $\tau$, the matrix $X$ is not square.
In $X$ replace each column for $b\in B_{int}^*$
with three consecutive identical columns. This yields a square matrix $\tilde X$.
For each $b\in B_{int}^*$, we now treat $b$ as three nodes which are three copies of itself:
we add three consecutive black nodes $b_1,b_2,b_3$ to the set $\bB$ (in clockwise order, at some point along the boundary).

Let $\Sigma$ be the set of bijections from triples of $\tau$ to $B_{int}^*$.
Given $\tau$ and a bijection $\alpha\in\Sigma$ we extend $\alpha$ to a 
bijection (pairing) $\pi$ from the set of all white nodes to the augmented set of black nodes,
by sending, for each triple $t=w_1w_2w_3$, the nodes $w_1,w_2,w_3$ to $b_1,b_2,b_3$ in that order, where $b=\alpha(t)$ (the pairs of $\tau$ are still paired in $\pi$). See Figure \ref{tripletobij}. 
\begin{figure}[htbp]
\begin{center}\includegraphics[width=3in]{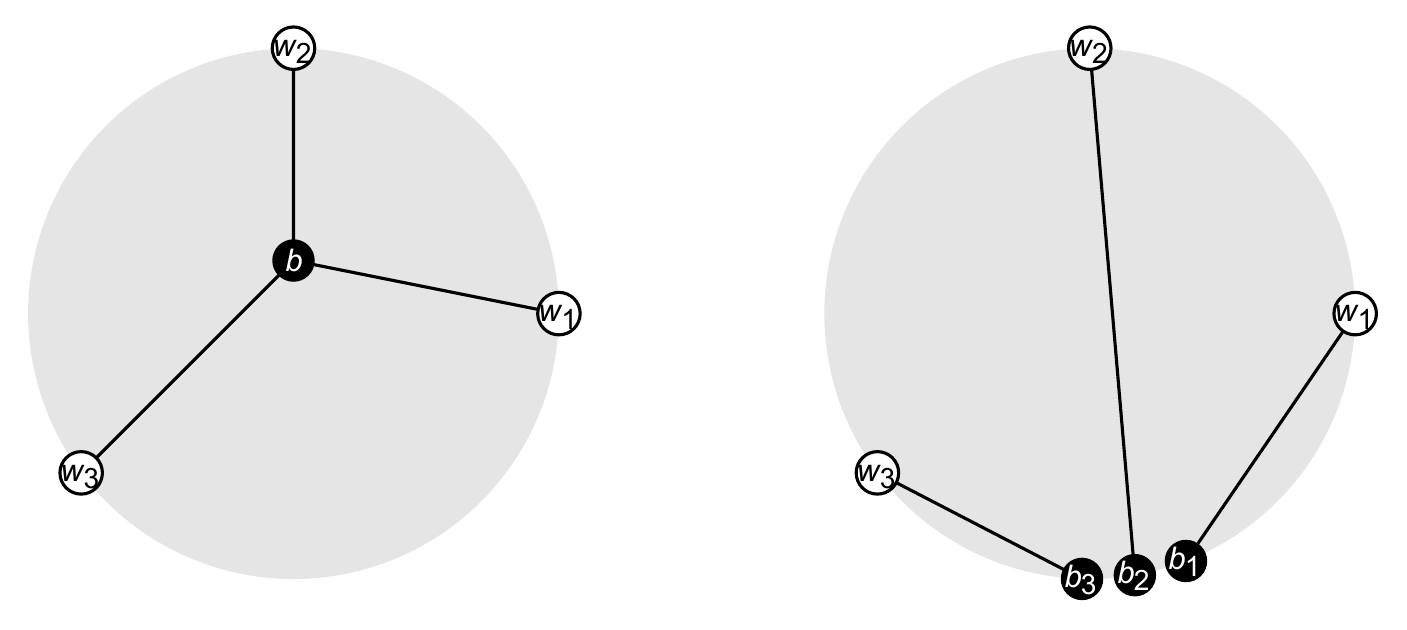}\end{center}
\caption{\label{tripletobij}Converting triples to pairings. On the left, a triple $t$ of $\tau$, with center vertex $b=\alpha(t)$. 
On the right, the corresponding pairing: we push $b$ to the boundary, replacing it with three new consecutive nodes
$b_1,b_2,b_3$ in clockwise order, connecting the $w_i$s to the $b_i$s. 
}
\end{figure}

Define
\be\label{Xtaudef}X_\tau := \sum_{\alpha\in\Sigma}\tilde X_\pi\ee
where $\pi=\pi(\alpha)$, and 
$\tilde X_\pi$ is defined for the pairing $\pi$ as in (\ref{Xtaupair}). Note that if we change $\pi$ to $\pi'$ by changing the bijection
between the $w_1,w_2,w_3$ and the $b_1,b_2,b_3$, then $\tilde X_{\pi'}=\pm \tilde X_\pi$ changes only by the sign of this bijection, since
the corresponding columns $b_1,b_2,b_3$ of $\tilde X$ are identical. 


In the above example, where $p=(b,w,w,w,w)$, the matrix $X$ is $4\times 2$.
 We have 
\begin{align}\nonumber X_{{\bf 1}1|234} &= X_{11}X_{22}X_{32}X_{42}\\
\label{Xtaupwwww}X_{{\bf 1}2|134} &= -X_{21}X_{12}X_{32}X_{42}\\
\nonumber X_{{\bf 1}3|124} &= X_{31}X_{12}X_{22}X_{42}\\
\nonumber X_{{\bf 1}4|123} &= -X_{41}X_{12}X_{22}X_{32}.
\end{align}
Note that if we relabel the second indices $2$ with $2,3,4$ (to correspond to their column in $\tilde X$) the sign is $(-1)^\pi$.

\subsection{Main result}

\begin{thm}\label{main}
We have $P_{\lambda} = \eps_K\sum_{\tau\in \Pi} \P_{\lambda,\tau}X_{\tau},$
where $\eps_K$ is a sign independent of $\lambda$, depending only on the choice of Kasteleyn signs. 
\end{thm}

As an example, combining table (\ref{pwwwwtable}) with (\ref{Xtaupwwww}) we get (\ref{Ppwwww}): for example
$$P_{\lambda_1} = X_{{\bf 1}3|124}+X_{{\bf 1}4|123} = X_{31}X_{12}X_{22}X_{42}-X_{41}X_{12}X_{22}X_{32}.$$

\begin{proof}[Proof of Theorem \ref{main}]
Suppose first that $|\bW|=|\bB|$, so there are no triples. 
Let $c$ have $p,q$ and $r$ white nodes of color $1,2,3$ respectively 
(and thus $p,q,r$ black nodes of color $1,2,3$ also). 
Let $W_i\subset\bW$ be the white nodes of color $i$
and $B_i\subset\bB$ be the black nodes of color $i$.  
By Lemma \ref{Zlem}, the partition function is $Z(c) = \eps_cZ_1Z_2Z_3$.  

We write (see Lemma \ref{Zlem} and (\ref{eee}))
\begin{align*}
Z(c) &= \eps_cZ_1Z_2Z_3\\
&= \eps_c\eps_K\delta_c\det K(\G_1)\det K(\G_2)\det K(\G_3).
\end{align*}

Thus with $X_i=X_{W_i}^{B_i}$ we have
\begin{align}\nonumber
\frac{Z(c)}{\Delta^3}&=\eps_c\eps_K\delta_c\det X_1\det X_2\det X_3\\\nonumber&=\eps_c\eps_K\delta_c\Big(\sum_{\sigma_1}(-1)^{\sigma_1}X_{i_1\sigma_1(i_1)}\dots X_{i_p\sigma_1(i_p)}\Big)\Big(\sum_{\sigma_2}(-1)^{\sigma_2}(\dots)\Big)
\Big(\sum_{\sigma_3}(-1)^{\sigma_3}(\dots)\Big)\\
\nonumber&= \eps_c\eps_K\delta_c\eta_c\sum_{\tau\in\Pi} (-1)^{\tau}X_{1\tau_1}\dots X_{n\tau_n}\Tr_{\tau}(c)\\
\label{XTr}&=\eps_K\sum_{\tau\in\Pi} X_\tau\Tr_{\tau}(c)
\end{align}
where in the third line we used the definition of $\eta_c$ and in the last line we used Lemma \ref{signs}.

Now (\ref{XTr}), which holds for each $c$, writes the function $Z$ as a linear combination of traces $\Tr_\tau$. 
Using (\ref{tausum}), we have
$$\Tr_{\tau} = \sum_\lambda\P_{\lambda,\tau}\Tr_{\lambda}.$$
Plugging this into (\ref{XTr}),
$$\frac{Z}{\Delta^3} = \eps_K\sum_{\tau\in \Pi} X_{\tau}\sum_{\lambda\in\Lambda}\P_{\tau,\lambda}\Tr_{\lambda}.$$
Interchanging the order of summation,
$$\frac{Z}{\Delta^3} = \eps_K\sum_{\lambda\in\Lambda}(\sum_{\tau\in \Pi} X_{\tau}\P_{\tau,\lambda})\Tr_{\lambda}.$$
Since the $\{\Tr_{\lambda}\}_{\lambda\in\Lambda}$ are independent functions (see Section \ref{tracedef} and \cite{Kuperberg}), this allows us to identify $P_{\lambda}$ with the coefficient of $\Tr_{\lambda}$ in this sum,
that is, $P_{\lambda}=\eps_K\sum_{\tau\in \Pi} X_{\tau}\P_{\tau,\lambda}$ as desired.

Now suppose we have $|\bW|> |\bB|$.
We reduce this case to the previous case, as follows. 
Recall the definition of $W^*$ from Section \ref{graphsection}; see Figure \ref{almost}. 
We add a spike to each $w\in W^*$, splitting
$w$ into two white vertices $w,w_s$ connected by a black vertex $b$ of degree $2$, so that neighbors of $w$ are now neighbors of $w_s$,
and $w$ is adjacent only to $b$. The new edges have weight $1$. See Figure \ref{augmented}, left.
We can then assume that $B_{int}^*$ consists of these new $b$ vertices.

We now add a ``gadget'' to $\G$ consisting of, for each $w\in W^*$, three new consecutive black nodes $b_1,b_2,b_3$,
and one white internal vertex, located on the same side of $w$ as the vertex $w'\in W^{**}$ associated to $w$. These are connected as shown in Figure \ref{augmented}. The edges
connecting $b_1,b_3$ to $w$ and $w''$ have weight $\eps$; other new edges have weight $1$. 
\begin{figure}[htbp]
\begin{center}\includegraphics[width=2.5in]{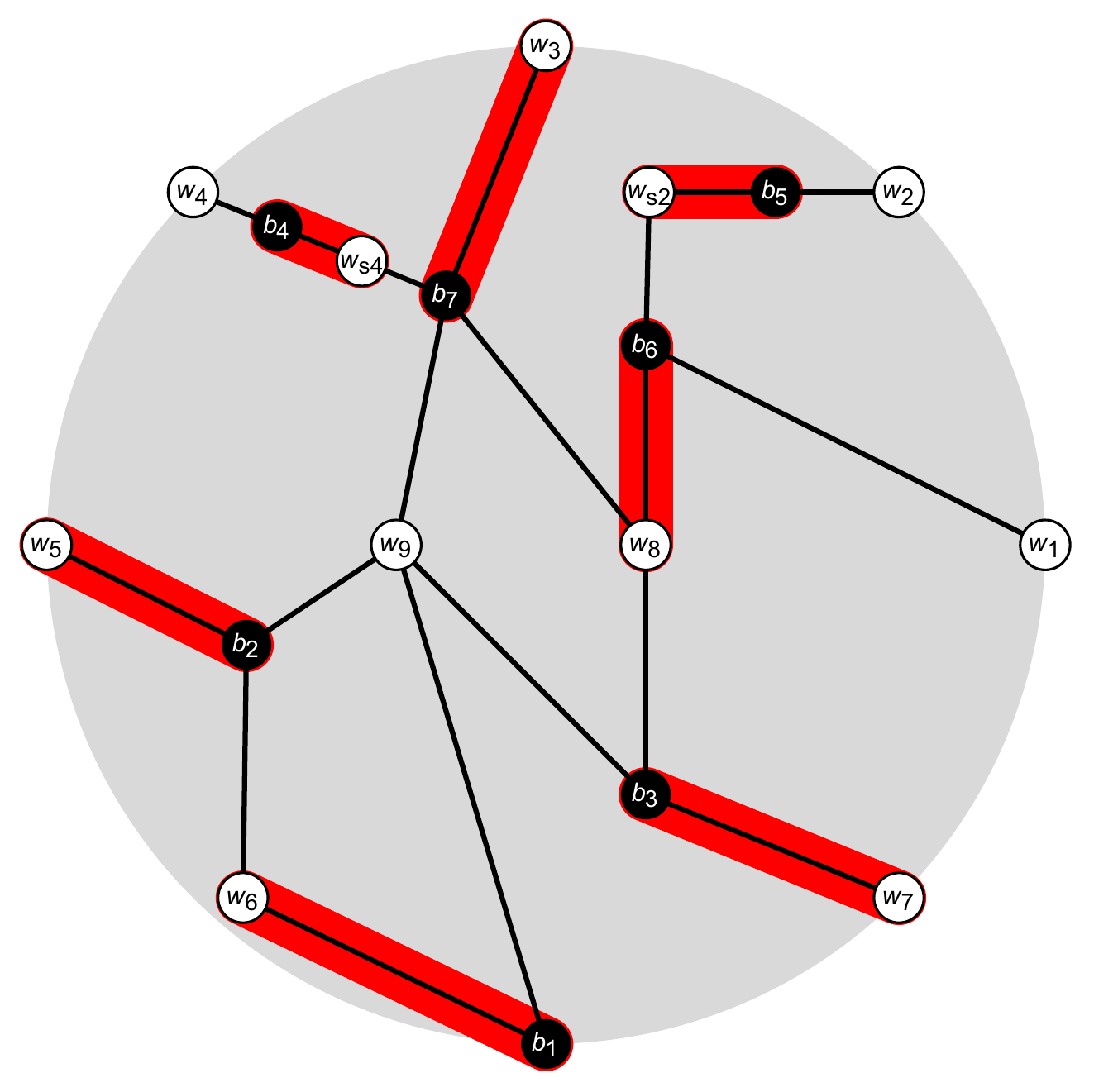}\hskip.5cm\includegraphics[width=2.5in]{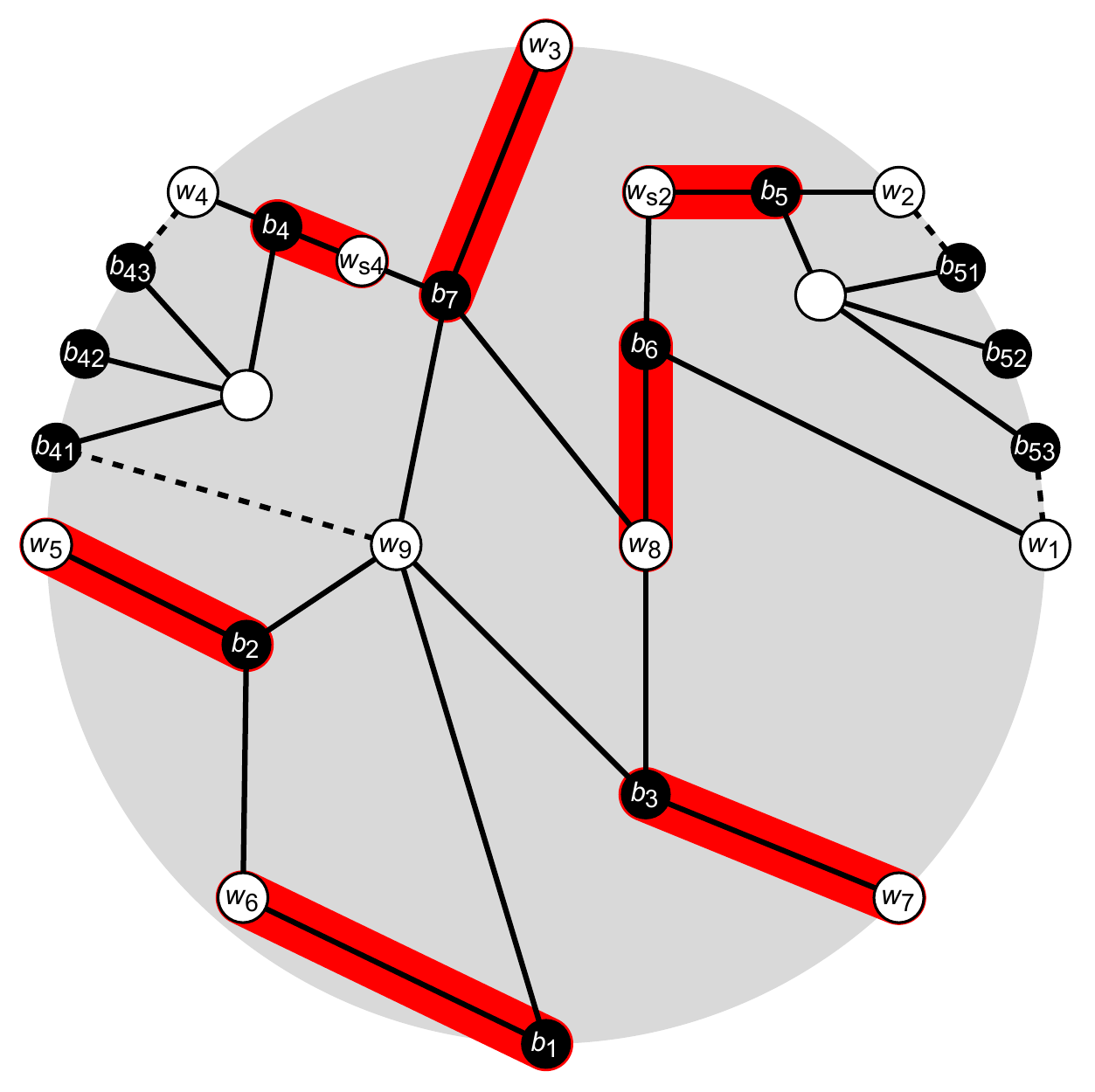}\end{center}
\caption{\label{augmented}We add spikes at $w_2,w_4$ splitting them into segments $w_2,w_{s2}$ and $w_4,w_{s4}$.
We then add gadgets to $B_{int}^*=\{b_4,b_5\}$ as shown. Dotted edges are the edges of weight $\eps$.}
\end{figure}

This augmented graph $\tilde\G$ has a dimer cover (extending the one of the figure to each gadget in a unique way).
Moreover in the limit $\eps\to 0$ the $X$ matrix $\tilde X$ for $\tilde\G$, which is square, is related to the $X$ matrix of $\G$:
For each $b\in B_{int}^*$ we replace the corresponding column of $X$ with three consecutive columns, labelled $b_1,b_2,b_3$,
each equal (in the limit $\eps\to0$) to the initial column, to get $\tilde X$.

We can now proceed as in the case $|W_\partial|=|B_{\partial}|$ above. 
Let $c$ be a coloring of the nodes of $\G$. We extend this to a coloring
$\tilde c$ of $\tilde\G$ by coloring each triple $b_1,b_2,b_3$ by colors $1,2,3$ in that order. From (\ref{XTr}), 
\be\label{Xtilde}\eps_K\frac{Z(\tilde c)}{\Delta^3} = \sum_{\rho\in\tilde\Pi} \tilde X_\rho \Tr_\rho(\tilde c).\ee
Here $\tilde \Pi$ consists of node pairings of the augmented set of nodes.
We group this sum into terms arising from each $\tau\in \Pi$: for each $\tau\in\Pi$ and bijection $\alpha\in\Sigma$ there is unique 
associated pairing $\rho\in\tilde\Pi$ with nonzero trace 
(since for each triple the white node of color $1$ must go to $b_1$, the white node of color $2$ must go to $b_2$,
and the white node of color $3$ must go to $b_3$). The sum in (\ref{Xtilde})  is
$$=\sum_{\tau\in\Pi} \sum_{\alpha\in\Sigma} \tilde X_\rho \Tr_\rho(\tilde c)$$
where $\rho=\rho(\tau,\alpha)$ is the associated pairing.

Recall that $X_\tau=\sum_\alpha \tilde X_{\pi}$ (see (\ref{Xtaudef}) where
$\pi$ pairs each triple $w_1,w_2,w_3$ with $b_1,b_2,b_3$ in the natural order. Let us compare $\tilde X_\rho$ with $\tilde X_\pi$: in $\tilde X_\rho$, the $w_1,w_2,w_3$ are paired with the $b_1,b_2,b_3$ so as to match the colors $\tilde c$;
the node among $w_1,w_2,w_3$ of color $1$ is paired with $b_1$ and so on. Thus
$\tilde X_\rho=\pm\tilde X_\pi$ where the sign is $+$ if and only if the permutation from $w_1,w_2,w_3$ to $b_1,b_2,b_3$ is even,
that is, if nodes $w_1,w_2,w_3$ are colored in counterclockwise order.
Note also that this is the sign of the contribution of this triple to $\Tr_\tau(c)$.
Thus $$\sum_\alpha \tilde X_{\rho}\Tr_\rho(\tilde c) = \sum_\alpha X_\pi \Tr_\tau(c)=X_\tau\Tr_\tau(c).$$
So the sum is 
$$\eps_K\frac{Z(c)}{\Delta^3} =\eps_K\frac{Z(\tilde c)}{\Delta^3}=\sum_{\tau\in\Pi} \sum_{\alpha\in\Sigma} \tilde X_\rho \Tr_\rho(\tilde c)=\sum_{\tau\in\Pi} X_\tau \Tr_\tau(c).$$
\end{proof}

\section{Web pairing matrices}

The reduction matrix $\P$ is related to two other matrices:
$$\M\P= \E$$ 
where $\M$ is the ``web pairing matrix'' and $\E$ is the ``extended web pairing matrix'',
both of which we now define.
Rows and columns of $\M$ are indexed by reduced webs and $\M_{r_1,r_2}$ is the trace
of the web drawn on the sphere with $r_1$ inside the disk and $r_2$ outside the disk (reflected in the bounding circle
from the inside to the outside the disk, so that the nodes of $r_2$ correspond to the nodes of $r_1$ in the same cyclic order).
In the $(w,b,w,b,w,b)$ example of Section \ref{wbwbwb} we have
$$\M = \begin{pmatrix}
27 & 9 & 9 & 3 & 9 & 24 \\
 9 & 27 & 3 & 9 & 3 & 24 \\
 9 & 3 & 27 & 9 & 3 & 24 \\
 3 & 9 & 9 & 27 & 9 & 24 \\
 9 & 3 & 3 & 9 & 27 & 24 \\
 24 & 24 & 24 & 24 & 24 & 72
 \end{pmatrix}.
$$

Rows of $\E$ are indexed by reduced webs and columns are indexed by $3$-partitions $\Pi$.
The entry $\E_{r,\pi}$ is the trace of the web drawn on the sphere with $r$ inside the disk and $\pi$
outside the disk, reflected as above. Since $\pi$ is not necessarily planar, the trace $\E_{r,\pi}$ may be negative or zero.
In the $(w,b,w,b,w,b)$-example we have
$$\E = \begin{pmatrix}
27 & 9 & 9 & 3 & 9 & 3 \\
 9 & 27 & 3 & 9 & 3 & 9 \\
 9 & 3 & 27 & 9 & 3 & 9 \\
 3 & 9 & 9 & 27 & 9 & 3 \\
 9 & 3 & 3 & 9 & 27 & 9 \\
 24 & 24 & 24 & 24 & 24 & 0
 \end{pmatrix}.
$$

\begin{thm}
Let $\M$ be the web trace matrix and $\E$ the extended web trace matrix.
Then $\M\P = \E$.
\end{thm}

\begin{proof}
Let $\lambda\in\Lambda$ and $\tau\in\Pi$. By definition $\E_{\lambda,\tau}$ is the trace of a graph drawn on the sphere,
which is $\lambda$ in the lower half sphere and $\tau$ in the upper half sphere, reflected as mentioned above.
We use skein relations in the upper half sphere to replace $\tau$ with a linear combination of planar webs,
$\tau = \sum_{\lambda'\in\Lambda}\P_{\lambda',\tau}r'$, and then apply the trace after extending by $\lambda$ in the lower half sphere, 
getting 
$$\E_{\lambda,\tau} = \sum_{\lambda'\in\Lambda} \M_{\lambda,\lambda'}\P_{\lambda',\tau}$$
as desired.
\end{proof}

We conjecture that $\M$ is invertible, in which case $\P=\M^{-1}\E$.
The matrices $\M,\E$ are the $\SL_3$ analogs of the $q=2$ ``meander matrix" $M$ and ``extended meander matrix $E$, which are relevant for $\SL_2$ connections, see \cite{KW11}. 
The meander matrix $M$ is indexed with planar pairings rather
than planar reduced webs, but otherwise has a similar definition: $M_{\pi,\pi'}$ is the trace of the $2$-web on the sphere formed from $\pi$
in the upper hemisphere and $\pi'$ in the lower hemisphere. Likewise for $E$. The trace in this case is just $2^c$ where $c$ is the number of components of the $2$-web.
It is shown in 
\cite{KS} that $M$ is invertible. In \cite {diFr}, di Francesco discusses a generalization of the meander matrix, called the $\text{SU}(n)$ meander matrix, and shows it is invertible for each $n$. We don't know if it is related to our matrix.

\section{Parallel crossings}\label{LGVscn}

The Lindstr\o{}m-Gessel-Viennot Theorem (see e.g. \cite{StanleyEC2}) is a determinantal formula for the number of pairwise disjoint monotone lattice paths in a planar lattice (with appropriate boundary conditions). It has a number of 
generalizations, notably by Fomin \cite{FominLERW} to the case of spanning forests
(or loop-erased walks), and by Kenyon and Wilson \cite{KW11} to the case of double dimer paths.

We give a generalization here for 
reduced webs consisting of parallel crossings. Remarkably our formula involves a product of two determinants,
rather than a single determinant.

Let $\G$ be a circular planar graph, with $2n$ nodes $v_1,\dots,v_{2n}$.
Suppose for each $i$ that $v_i$ and $v_{2n+1-i}$ have opposite type: one is white and the other black.
Consider the reduced web $\lambda_{\tiny{\|}}$ which pairs $i$ to $2n+1-i$, that is, represents the 
``parallel" crossing of a rectangle (see Figure \ref{parallel}).

\begin{figure}[htbp]
\begin{center}\includegraphics[width=3in]{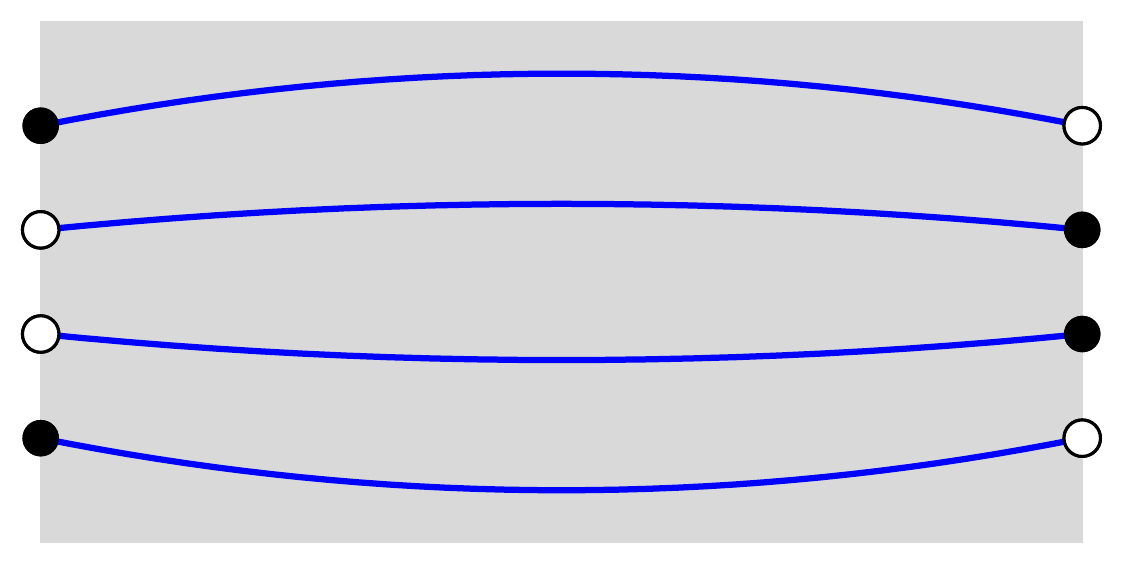}\end{center}
\caption{\label{parallel}The ``parallel crossing" reduced web.}
\end{figure}

Let $W_1,W_2$ be the set of indices of white nodes on the left and right, respectively, and likewise
$B_1,B_2$ the set of black nodes on the left and right. Let $X_{W_1}^{B_2}$ be the submatrix of $X$ with 
rows indexing $W_1$ and columns indexing $B_2$.

\begin{thm}\label{LGV}
We have $P_{\lambda_{\tiny{\|}}} = \det X_{W_1}^{B_2}\det X_{W_2}^{B_1}.$
\end{thm}

For example, see Figure \ref{Rt42}, second panel, for which (after a rotation) $W_1=\{1,2\}$ and $B_1=\emptyset$ and
whose polynomial is $\det X$.
See also the third configuration of Figure \ref{Rt61}, whose polynomial is $X_{2,3}\begin{vmatrix}
X_{1,1}&X_{1,2}\\X_{3,1}&X_{3,2}\end{vmatrix}$
(we changed indices to match those of the figure.)

If $c\equiv 1$, the probability of the parallel crossing, for the natural measure on reduced webs, 
is $Pr(\lambda_{\tiny{\|}}) = \frac{\det X_{W_1}^{B_2}\det X_{W_2}^{B_1}}{\det X}.$

\begin{proof}[Proof of Theorem \ref{LGV}]
By Theorem \ref{main} we need to compute $\P_{\lambda,\tau}$ for $\lambda=\lambda_{\tiny{\|}}$ and any pairing $\tau$. 
For clarity draw $\G$ in a rectangle so that nodes $1,\dots,n$ are on the left side and nodes $n+1,\dots,2n$ are on the right side.

Let $\tau\in\Pi$ be a pairing of nodes. 
From the definition of $\P$, we need to reduce $\tau$ to a linear combination of planar 
pairings using skein relations and see when the parallel crossing 
$\lambda$ occurs in this reduction.
First suppose that $\tau$ pairs two nodes $w,b$ which are both on the left side of the rectangle.
By induction on the distance from $w$ to $b$ we will show that no reduction of $\tau$ contains
$\lambda$, and thus $\P_{\lambda,\tau}=0$. Assume that no two nodes on the left side between $w$ and $b$ are paired:
there is no nested pairing inside $wb$.
If $w,b$ are adjacent, they never participate in any skein relation, so $w,b$ are paired in any planar reduction of $\tau$,
so $\P_{\lambda,\tau}=0$. If $w,b$ are not adjacent, we ``comb" the web $\tau$, isotoping it 
so that there are no crossings of pairs in $\tau$ between the strand $wb$ and the left boundary, as shown in Figure \ref{combed}.  That is, move all crossings enclosed between the strand $wb$ and the left boundary across the strand $wb$.
\begin{figure}[htbp]
\begin{center}\includegraphics[width=3in]{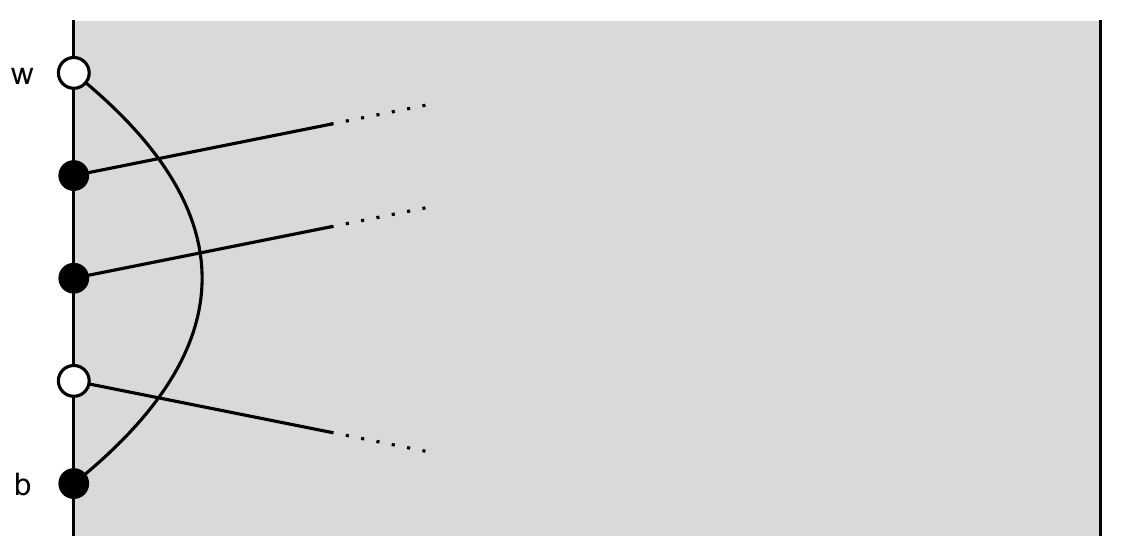}\end{center}
\caption{\label{combed}Draw the web so that there are no crossings of strands between the left boundary and the strand $wb$.}
\end{figure}

Suppose $w$ lies above $b$ on the left boundary. If the first node below $w$ is also white, say $w'$,
then applying a skein relation to the local crossing between $w$ and $w'$ results in two webs; the first connects $w'$ to $b$
and can be eliminated by induction. The second has a double-Y connecting $w,w'$ to an inner black vertex $b'$.
These three vertices persist under any further skein reductions, showing that $\P_{\lambda,\tau}=0$. 
We can therefore assume the first node below $w$ is black, say $b_1$. 
There are now three cases: the next node below $b_1$ is $b$, or another black node $b_2$, or a white node $w_1$.
If the next node below $b_1$ is $b$, then apply the skein relation to the intersection of the strand $wb$ with the strand from $b_1$; neither resulting web contributes to $\P$ as above.
Suppose the next node below $b_1$ is a black node $b_2\ne b$; the white case is similar. 
Let $x_1,x_2$ denote the crossings between $wb$ and the strand starting at $b_1$, and between $wb$ and the strand starting at $b_2$.
Apply a skein relation to the crossing $x_1$.
Replacing $x_1$ with a parallel crossing produces an adjacent pair which persists under all further reductions as above.
So we must replace the crossing $x_1$ with a double-Y. Now apply the skein relation to the crossing $x_2$. 
Replacing it with a parallel crossing results in $b_1,b_2$ being connected to an internal vertex $w'$ and these three persist
for all future reductions. So we must replace $x_2$ with a double-Y to get a nonzero contribution. The edge between these
two double-Ys (the thick blue edge in Figure \ref{combedthick})
will persist under all further reductions unless it is, at some point, part of a quad move (Figure \ref{skeinrels}, last line) involving the face immediately to its right.
\begin{figure}[htbp]
\begin{center}\includegraphics[width=3in]{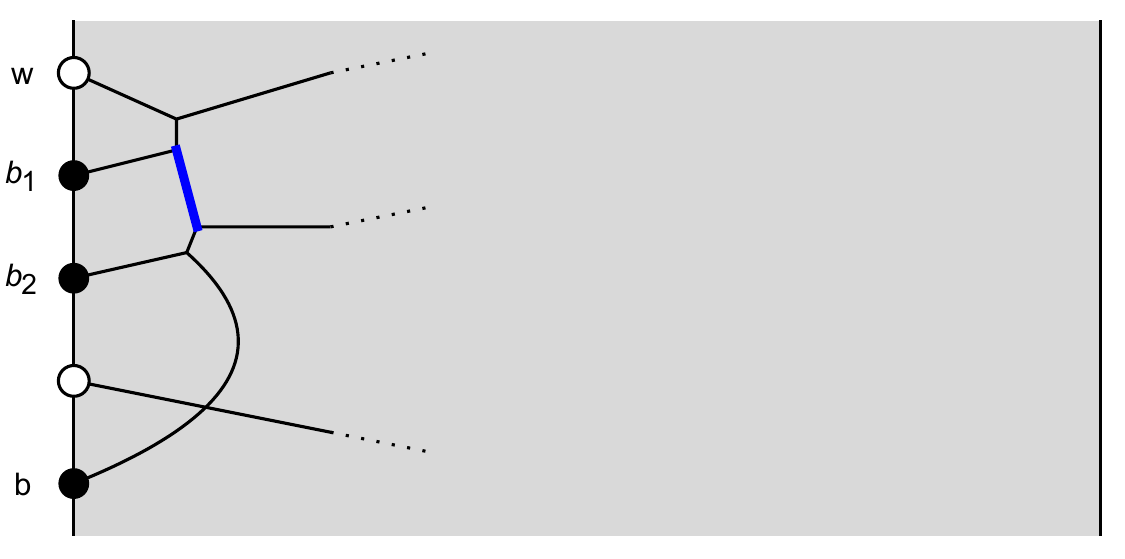}\end{center}
\caption{\label{combedthick}Resolving the crossings at $x_1,x_2$.}
\end{figure}

However the two possible resolutions of this quad result, in one case in an adjacent pair $wb_1$, 
and in the other case a pair of adjacent black nodes $b_1b_2$
connected to an internal white vertex; both of these cases we can eliminate as above.
This completes the proof that $\P_{\lambda,\tau}=0$ when $\tau$ pairs two nodes $w,b$ both on the same side. 

Now suppose that $\tau$ only pairs nodes
from opposite sides. We claim that $\P_{\lambda,\tau}=1$.
If $\tau$ is planar (has no crossings), it must be $\lambda$. 
Consider the pairs in $\tau$ with white nodes on the left; if none of these cross each other, and the same is true for the pairs with black nodes on the left,
then $\tau$ is planar. So if $\tau$ is nonplanar it must have two pairs which cross, and whose left nodes have the same type (either both black or both white), say white. 
In this case we can use the skein relation to replace this crossing with a parallel uncrossing, and another web which connects
two white vertices $w_1,w_2$ on the left to a black vertex $b$ in the interior. We claim by induction on the distance between $w_1,w_2$ that such a web contributes zero. 
If $w_1,w_2$ are adjacent, see as above. If not, comb so that there are no crossings of other strands in the region between $w_1,b,w_2$ and the left boundary. Further assume that strands for all nodes between $w_1,w_2$ exit this region between $w_1$ and $b$, that is, cross $w_1b$. Then we can argue exactly as in the previous case of a pairing between a black and white vertex on the left by making the strand $w_2b$ very small $b$ acts as if it is on the left boundary.

Consequently every crossing between pairs with the same type on the left must resolve to a \emph{parallel uncrossing},
in order to contribute nontrivially to $P_{\lambda,\tau}$. 
Resolving all such pairs results in the web $\lambda$, with a sign $+$. Therefore $\P_{\lambda,\tau}=1$ for all pairings
connecting nodes on opposite sides. 
Therefore $\sum_{\tau}\P_{\lambda,\tau}X_\tau= \sum_{\tau\in\Pi'}X_\tau$ where $\Pi'$ is the set of $\tau$'s
pairing nodes on opposite sides. Using the definition of $X_\tau$ from (\ref{Xtaupair}) this completes the proof.
\end{proof}

The above proof 
generalizes (thanks to Pasha Pylyavskyy for suggesting this) to crossings with ``crossbars", as illustrated in Figure \ref{crossinggeneral}:
take $n$ parallel strands as in Figure \ref{parallel}, and draw any nonnegative number of crossbars between adjacent strands
with the proviso that the crossbars be \emph{interlaced}: between two crossbars connecting strands $i$ and $i+1$ there is exactly one crossbar from $i-1$ to $i$ and one from $i+1$ to $i+2$. 
Note that this produces a reduced web, and the vertex types are determined on each connected component up to a global type flip.
 \begin{figure}[htbp]
\begin{center}\includegraphics[width=3in]{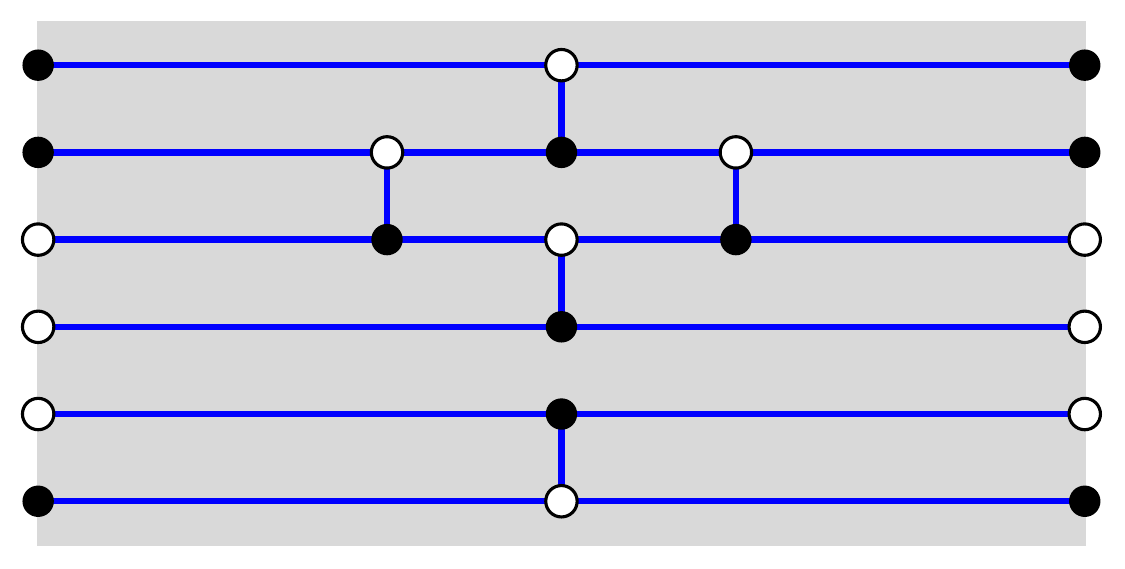}\end{center}
\caption{\label{crossinggeneral}A crossing with crossbars (a crossbar is a vertical segment joining two horizontal lines).}
\end{figure}
As above let $W_1,W_2$ be the white nodes on the left and right and $B_1,B_2$ the black nodes on the left and right.

\begin{thm}For a reduced web $\lambda$ which is a crossing with crossbars as defined above, $P_\lambda=\det X_{W_1}^{B_2}\det X_{W_2}^{B_1}.$
\end{thm}

\begin{proof}
The first part of the proof, showing that $\tau$ cannot pair two nodes on the same side, is identical to the first part of the
proof of Theorem \ref{LGV} above.
So we may suppose that $\tau$ only pairs nodes on opposite sides. In this case we claim that $\P_{\lambda,\tau}=1$. 

Consider a pair of crossing strands of $\tau$ with black nodes on the left and white on the right, or vice versa.
Applying the skein relation, turning a crossing into a double-Y and a parallel uncrossing, the double-Y cannot contribute to $\P_{\lambda,\tau}$ as discussed in the previous proof.
So we must resolve this crossing into two parallel strands (note that this does not introduce any sign change).
We resolve all crossings between strands with the both left nodes black or both left nodes white.
 
We can now suppose that all strands of $\tau$ with white nodes on the left are disjoint, and all strands
of $\tau$ with black nodes on the left are disjoint. Note that there is a unique such pairing $\tau$.
Each crossing in $\tau$ involves two strands with different types on the left.
The resolution of this crossing must give a double-Y, since otherwise there would be a strand connecting two vertices on the left. Resolving all crossings this way gives $\lambda$ exactly.

In conclusion the pairings $\tau$ that contribute to $\P_{\lambda,\tau}$ are exactly 
those that pair nodes on opposite sides, and the contribution is always $+1$. 
In particular 
$\sum_{\tau} \P_{\lambda,\tau}X_{\tau}$ is the form in the statement.
\end{proof}

\section{Application: triangular honeycombs}\label{honeycomb}

As another application of Theorem \ref{main}, 
we can compute the $P$-polynomial (and probability) for the order-$n$ 
triangular honeycomb web $T_n$ of Figure \ref{bigweb3} and the order-$n$ honeycomb $T'_n$ of Figure \ref{bigweb4prime}. In the case of $T_n$ and Figure \ref{bigweb3}
we used a different convention for the black nodes: ordering them counterclockwise. This
makes for a slightly less messy indexing in the statement of Theorem \ref{PTn1}.

\subsection{The triangular honeycomb $T_n$} \label{Tn}

\begin{thm}\label{PTn1}
We have 
$$P_{T_n} = \det(\{X_{i,j+n}\}_{i,j=1..n})\det(\{X_{i+n,j+2n}\}_{i,j=1..n})\det(\{X_{i+2n,j}\}_{i,j=1..n}).$$
\end{thm}

For example for the order-$3$ example this is 
$$P_{T_3}  = \det\begin{vmatrix}X_{1,4} & X_{1,5} & X_{1,6} \\
 X_{2,4} & X_{2,5} & X_{2,6} \\
 X_{3,4} & X_{3,5} & X_{3,6} \\
 \end{vmatrix}
 \det\begin{vmatrix}X_{4,7} & X_{4,8} & X_{4,9} \\
 X_{5,7} & X_{5,8} & X_{5,9} \\
 X_{6,7} & X_{6,8} & X_{6,9} 
 \end{vmatrix}
  \det\begin{vmatrix}X_{7,1} & X_{7,2} & X_{7,3} \\
 X_{8,1} & X_{8,2} & X_{8,3} \\
 X_{9,1} & X_{9,2} & X_{9,3} 
 \end{vmatrix}.$$
 
 \begin{figure}[htbp]
\begin{center}\includegraphics[width=4in]{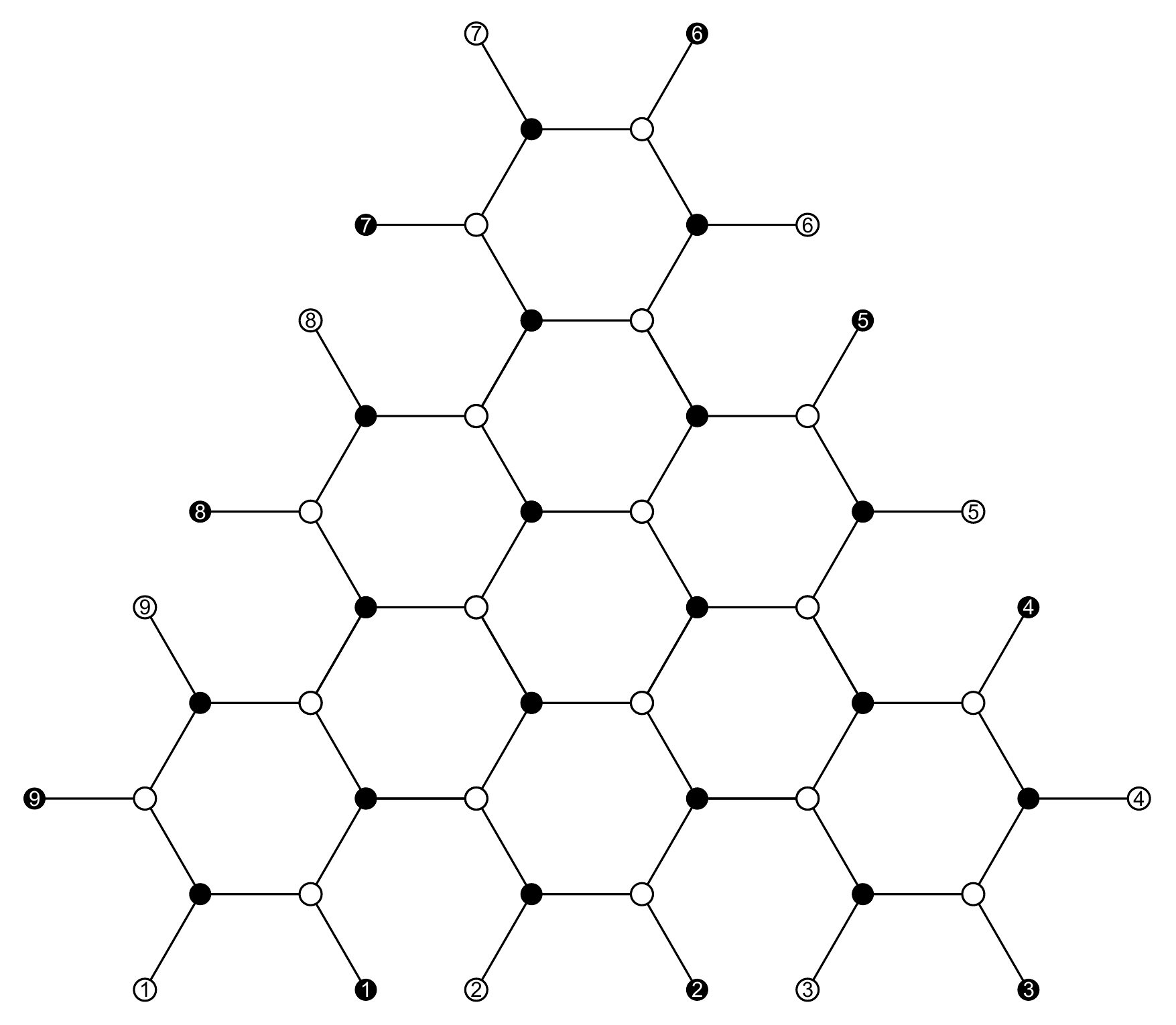}\end{center}
\caption{\label{bigweb3}Triangular honeycomb web $T_n$ of order $3$.}
\end{figure}

\begin{proof}[Proof of Theorem \ref{PTn1}]
Let $\lambda$ denote the reduced web $T_n$. 
Let $I_1,I_2,I_3$ be the three sides of $T_n$, with $I_1$ being the lower side as in the Figure:
$I_1$ consists of white and black nodes $\{1,\dots,n\}$, $I_2$ consists of white and black nodes
$\{n+1,\dots,2n\}$ and $I_3$ the remaining nodes. 

The proof is similar to the proof of Theorem \ref{LGV} above. 
Consider a pairing $\tau\in\Pi$. In order to have 
$\P_{\lambda,\tau}\ne0$, we claim that $\tau$ must pair all white nodes in $I_i$ to black nodes in $I_{i+1}$, with cyclic indices.
Suppose not; suppose first that $\tau$ pairs some white node $w\in I_1$ with a black node $b$ on the same side: $b\in I_1$.
In this case an induction argument identical to that in paragraphs 2 and 3 of the proof of Theorem \ref{LGV}
implies that $\P_{\lambda,\tau}=0$. Secondly suppose $\tau$ pairs a white node $w\in I_1$ with a black node $b\in I_3$.
Suppose $b$ is not the lower black vertex of $I_3$; let $w'$ the white node below it in $I_3$.
Comb strands left of this pair $bw$ as in Figure \ref{combed2}, left, so there are no crossings left of strand $wb$. \begin{figure}[htbp]
\begin{center}\includegraphics[width=1.5in]{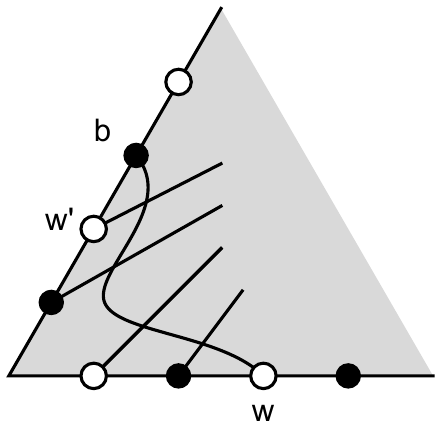}\hskip1cm\includegraphics[width=1.5in]{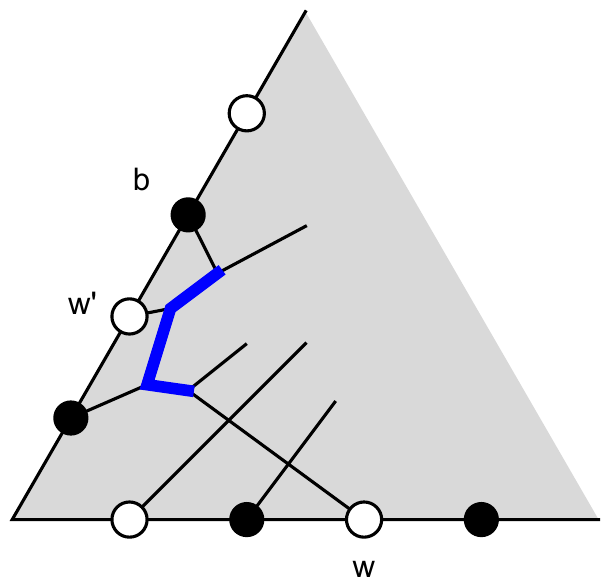}\end{center}
\caption{\label{combed2}Left panel: combing the strand from $w$ to $b$. Right panel: after applying the basic skein relation to the 
two crossings near $b$.}
\end{figure}
Apply skein relations to the strand crossings between $bw$ and the strand from $w'$ and the strand from the black vertex below it, as in the Figure, right panel. In order to obtain $\lambda$,
the thick edges of Figure \ref{combed2}, right, must be part of a quad face in some subsequent skein reduction (otherwise the outer face between $b$
and $w'$ would remain of degree $3$). Either reduction of this quad face results in a web which cannot be reduced to $\lambda$. If $b$ happens to be the lower black vertex of $I_3$, $w$ can not be the left-most white node of $I_1$ (otherwise
$\P_{\lambda,\tau}=0$) and 
then the same argument using the last two crossings near $w$ (with the roles of $b$ and $w$ switched) applies.

If for each $i$ every white node in $I_i$ connects to a black node in $I_{i+1}$, then we claim $\P_{\lambda,\tau}=(-1)^n$. 
As in the previous proof, consider first crossings between pairs connecting white nodes in $I_i$ to black nodes in $I_{i+1}$. 
Any such crossing must resolve into parallel strands. Resolving all these crossings
leads to the pairing $\tau_{\text{cross}}$ of Figure \ref{cross3}. Each crossing of $\tau_{\text{cross}}$ (of which there are $3n^2$)
must then resolve to a double-Y, and thus $\P_{\lambda,\tau}=\P_{\lambda,\tau_{\text{cross}}}=(-1)^{3n^2}=(-1)^n$.
This implies the statement.

\begin{figure}[htbp]
\begin{center}\includegraphics[width=3in]{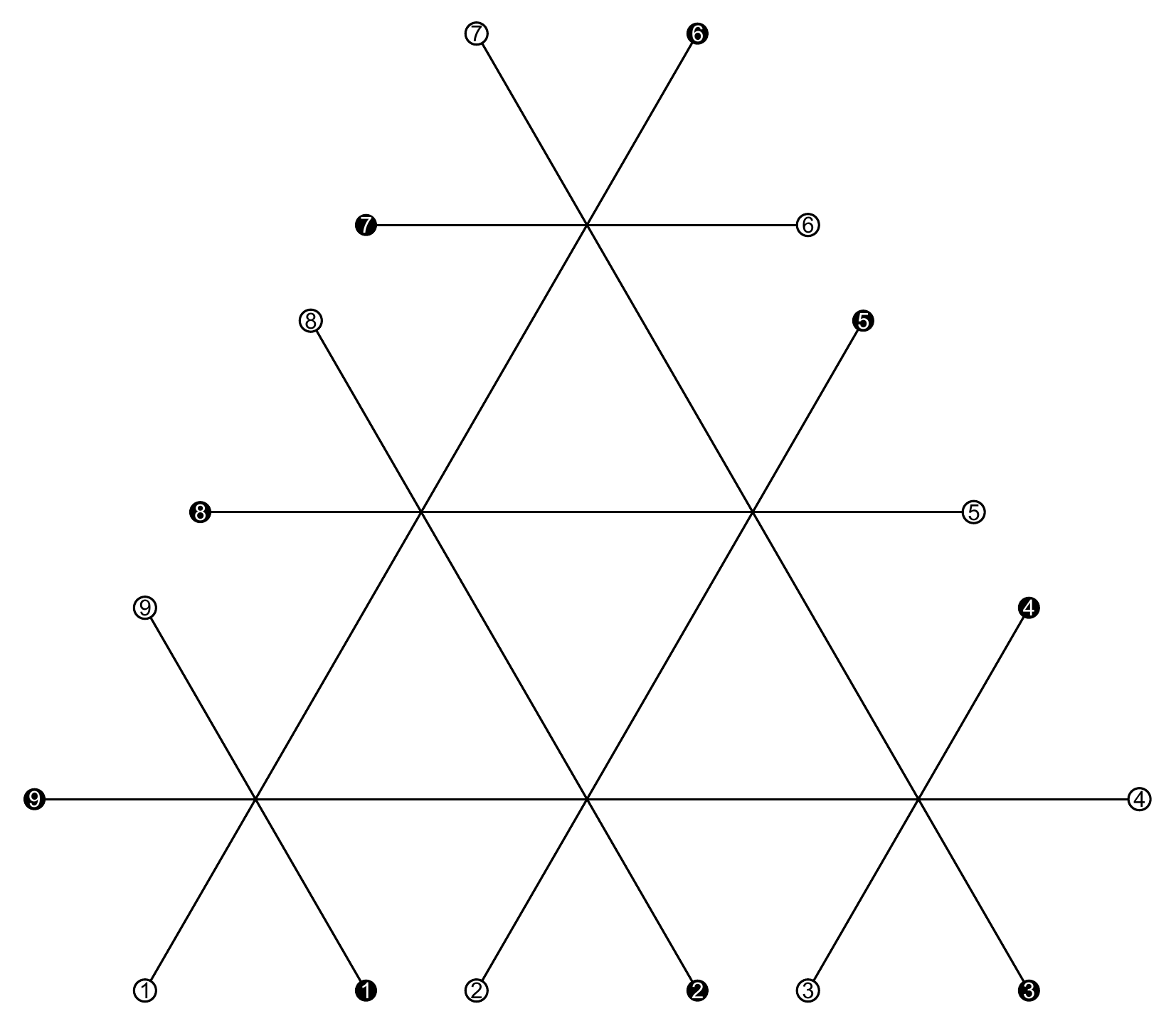}\end{center}
\caption{\label{cross3}Basic crossing diagram for the order-$3$ honeycomb web.}
\end{figure}

\end{proof}

For boundary colors $c\equiv 1$, the partition function is $Z(c)=Z(Z_{int})^2 = \Delta^3\det X$, which is the denominator in the expression
for the probability of $T_n$. 
We thus have, for the measure $\mu_c$,
\be\label{PrTn}Pr(T_n) = \frac{P_{T_n}\Tr_{T_n}(1)}{\det X}.\ee
Although we don't have an exact expression for $\Tr_{T_n}$, 
Baxter \cite{Baxter} computed the asymptotic growth rate of the trace of the honeycomb web to be
$$\lim_{n\to\infty}\frac1{n^2}\log Tr_{T_n}(1) = \frac12\log(\frac{3\Gamma(\frac13)^3}{4\pi^2}).$$

\subsection{The triangular honeycomb $T_n'$} \label{Tnprime}

For $T_n'$, all $3n$ nodes are white. They are grouped into subintervals $W^1=\{1,\dots,n\}$, $W^2=\{n+1,\dots,2n\}$ and
$W^3=\{2n+1,\dots,3n\}$. The matrix $X$ is $3n\times n$. Let $B$ denote the columns of $X$.
Let $\det X_{W^i}^B$ be the determinant of the maximal minor of $X$ corresponding to rows $W^i$.

\begin{thm}
We have 
$P_{T'_n} = \det X_{W^1}^B\det X_{W^2}^B\det X_{W^3}^B.$
\end{thm}

\begin{figure}[htbp]
\begin{center}\includegraphics[width=3in]{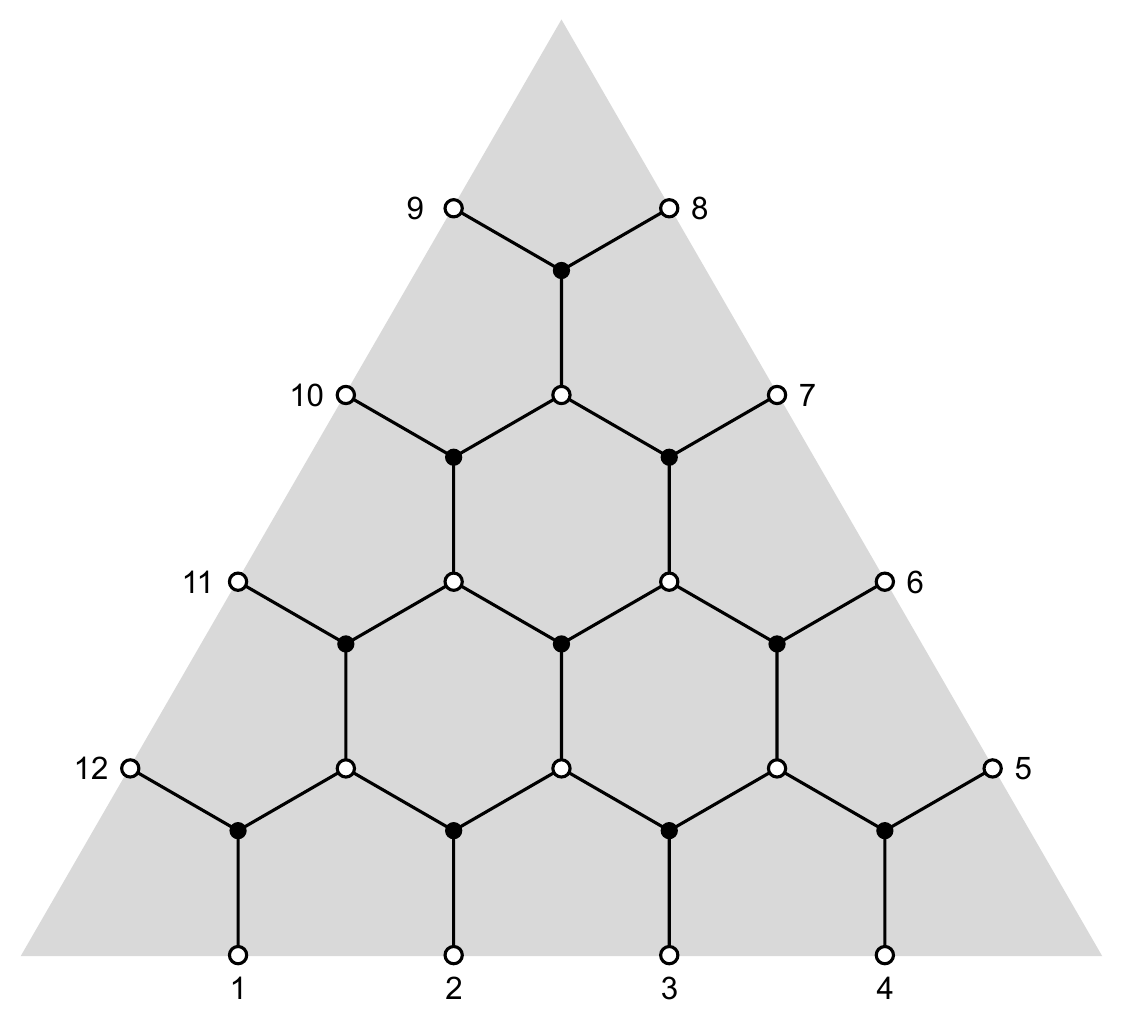}\end{center}
\caption{\label{bigweb4prime}Triangular honeycomb web $T'_n$ of order $4$.}
\end{figure}

\begin{proof}
Let $\lambda=T_n'$. 
By Theorem \ref{main}, we need to find $3$-partitions $\tau$ (partitions into parts of size $3$) of $\{1,\dots,3n\}$ 
which contribute to $\P_{\lambda,\tau}.$
We prove that $\tau$ contributes if and only if each triple of $\tau$ contains one element of each of $W^1, W^2, W^3$.

First suppose a part of $\tau$ contains two nodes $w,w'$ of the same subinterval, say $W^1$. We then claim that $\P_{\lambda,\tau}=0$.
We apply skein relations to write $\tau$ as a formal linear combination of planar reduced webs. We argue by
induction on the number of nodes separating $w$ and $w'$. If $w,w'$ are adjacent, and connected to internal vertex $b$,
then in any planar resolution $w, b,w'$ will be connected in the same way $w\sim b\sim w'$, that is, no skein relation changes these connections.
Thus $\P_{\lambda,\tau}=0$. If $w,w'$ are not adjacent, first ``comb'' the crossings so that there are no crossing of strands
between $w,b,w'$ and the boundary, as in Figure \ref{combed3}. 
\begin{figure}[htbp]
\begin{center}\includegraphics[width=1.5in]{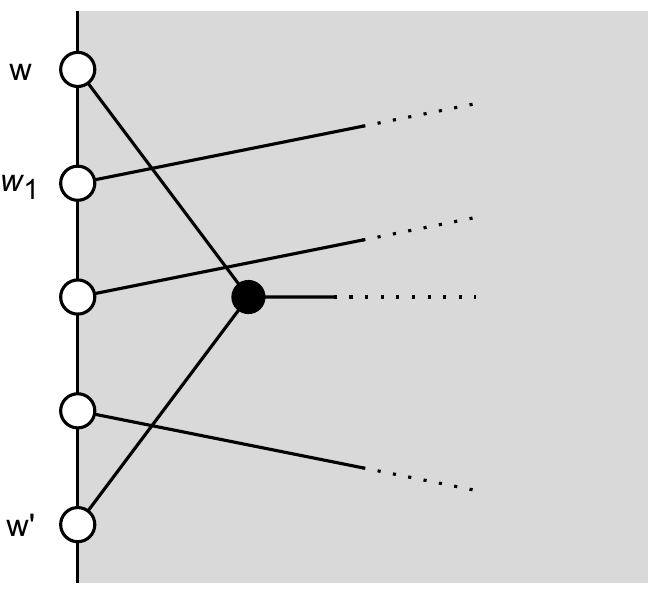}\end{center}
\caption{\label{combed3} Combing of $\tau$ between $w$ and $w'$.}
\end{figure}
Let $w_1$ be the first node adjacent to $w$, between $w$ and $w'$. Now either resolution of the crossing of the strand starting at $w_1$ and the path $wbw'$ results in a web with two connected 
nodes closer then $w,w'$ are. The completes the proof that $\P_{\lambda,\tau}\ne0$ only if each part of 
$\tau$ contains nodes from each of $W^1,W^2,W^3$. 

Now suppose each part of $\tau$ contains a point in each of $W^1,W^2,W^3$.
We show that $\P_{\lambda,\tau}=\pm1$ with the sign given by the above determinantal formula.
We show this by induction on $n$, the number of parts of $\tau$. If $n=1$, the conclusion holds.
Suppose we have constructed a copy of $T_{n-1}'$ from the first $n-1$ parts of $\tau$.
We embed the last part $ijk$ of $\tau$ so that its central vertex $b$ is near the lower left corner of the triangle.
This means the strand from $i$ to $b$ crosses $i-1$ vertical strands on the lower level of $T_{n-1}'$,
the strand from $j$ crosses $j-n-1$ strands on side $2$ and $n-1$ strands on side $1$, and the strand from $k$ crosses
$3n-k$ strands on side $3$.  See Figure \ref{addone}. 
\begin{figure}[htbp]
\begin{center}\includegraphics[width=3.5in]{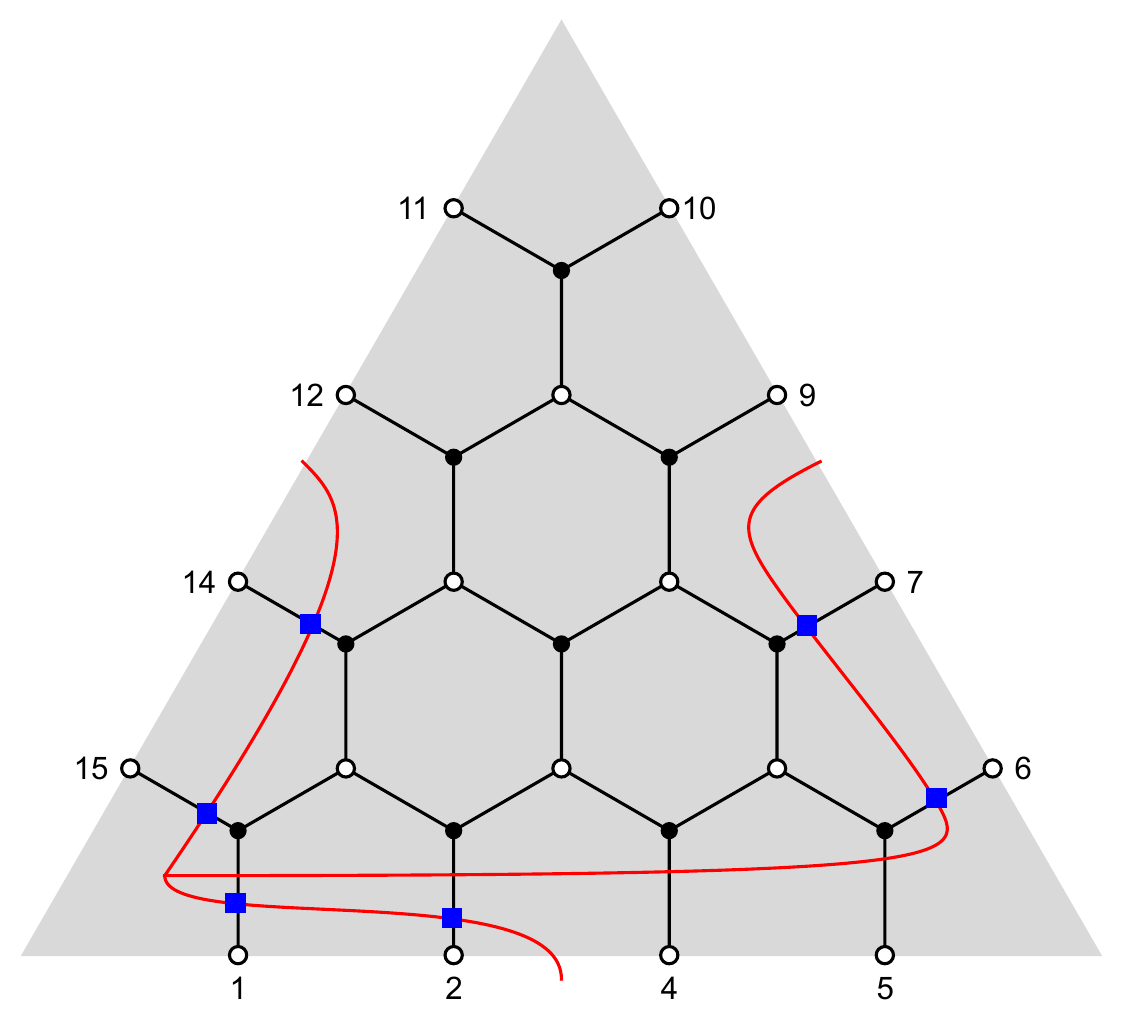}\end{center}
\caption{\label{addone}Induction step.}
\end{figure}
These crossings all necessarily resolve as parallel crossings except for the $n$ crossings
of strand $j$ along side $1$, which result in $n-2$ more hexagonal faces and thus creates $T_n'$. 

To compute the sign contribution, take a $3$-partition $\tau$ and bijection $\sigma$ from the parts of $\tau$ to $B_{int}^*$.
Compare $\tau$ with the partition $\tau_0$ into parts $(i,n+i,2n+i)$
with $i$ from $1$ to $n-1$, and bijection sending $(i,n+i,2n+i)$ to $b_i$, the $i$th element of $B_{int}^*$. 
We can order the parts of $\tau$ according to $\sigma$. Then $\tau$ is defined by $3$ permutations $\sigma_1,\sigma_2,\sigma_3$
where $\sigma_i$ orders the parts in $W^i$. Now $\tau$ can be drawn in the triangle so that it agrees with the drawing of 
$\tau_0$ except for
a narrow band along each edge where the strands cross according to the permutation $\sigma_i$. We can reduce $\tau$ 
by first uncrossing all the strands in the bands to make a copy of $\tau_0$, then reducing $\tau_0$. 
Thus the sign from the contribution of $\tau$ equals $(-1)^{\sigma_1+\sigma_2+\sigma_3}$ times the sign of the contribution
of $\tau_0$. 

This is the same sign as that given by the product of determinants in the statement. 
\end{proof}

\section{Scaling limit}\label{scalinglimitscn}

Let $\H$ denote the upper half plane and suppose $\G_\eps=\eps\Z^2\cap \H,$
the part of $\eps\Z^2$ in the closed upper half plane. Let $z_1,z_2\in \R$ be distinct points
on the boundary of $\H$. In $\G_\eps$ let $w_\eps,b_\eps$ be, respectively, a white and black vertex
on the lower boundary of $\G_\eps$,
and converging to $z_1,z_2$ as $\eps\to0$. 
We have \be\label{Xsclim}X_{w_\eps,b_\eps} = \frac{2\eps}{\pi(z_2-z_1)} + O(\eps^2),\ee
see \cite{Kenyon.ci}.

Let $w_1<b_1<w_2<b_2$ be four points in $\R$. Let $w_1^\eps,b_1^\eps,w_{2}^\eps,b_{2}^\eps$ be white and black vertices of $\G_{\eps}$ converging to the $w_i,b_i$ as above. Then with $c\equiv 1$ a random reduced web will connect $w_1$ with $b_1$ and $w_2$ with $b_2$ with
probability (see (\ref{4ptprobs}) above)
$$\Pr(\lambda_1) = \frac{X_{1,1}X_{2,2}}{X_{1,1}X_{2,2}-X_{1,2}X_{2,1}}~~ \stackrel{\eps\to0}{\longrightarrow}~~ \frac{(b_2-w_1)(w_2-b_1)}{(b_2-b_1)(w_2-w_1)}.$$
Note that this limiting quantity is M\"obius invariant; it only depends on the cross ratio of the four boundary points. More generally
this quantity is conformally invariant, if we replace the upper half plane with a Jordan domain with four marked boundary points, approximated by $\eps\Z^2$ as $\eps\to0$ in an appropriate sense \cite{Kenyon.ci}.

Let $w_1<b_1<\dots<w_{3n}<b_{3n}$ be $6n$ points on the boundary of $\H$. Let $w_1^\eps,b_1^\eps,\dots,w_{3n}^\eps,b_{3n}^\eps$ be white and black vertices of $\G_{\eps}$ converging to the $w_i,b_i$ as above.
For boundary colors $c\equiv 1$ the probability of the honeycomb web $T_n$ of Figure \ref{bigweb3},
given by expression (\ref{PrTn}), takes a nice form in the $\eps\to0$ limit since the determinants are of Cauchy matrices. 
Letting $I_1=\{1,\dots,n\},~~I_2=\{n+1,\dots,2n\}$ and $I_3=\{2n+1,\dots,3n\}$, 

\be\label{PrTn2}Pr(T_n) \to \frac{\prod_{\text{$(i,j)$ not in $(I_1,I_2),(I_2,I_3)$ or $(I_3,I_1)$}}(w_i-b_j)}{\prod_{\text{$i, j$ in diff't parts}}(w_i-w_j)(b_i-b_j)}\Tr_{T_n}(1).\ee
See for example (\ref{beetle}).
This case is special in that the numerators and denominators factor into linear parts; this factorization does not hold for more general
boundary conditions.

\section{Appendix: $n=6$ cases}

When $n=6$, with three white and three black nodes there are, up to symmetry, three possible type patterns. 
For each pattern, $\Lambda$ consists in six classes.
The other case with $n=6$ consists of all nodes of the same type.

\subsection{Type pattern $(w,b,w,b,w,b)$}\label{wbwbwb}

When $p=(w,b,w,b,w,b)$, the six reduced web classes are shown in Figure \ref{Rt61}.
\begin{figure}[htbp]
\begin{center}\includegraphics[width=6.5in]{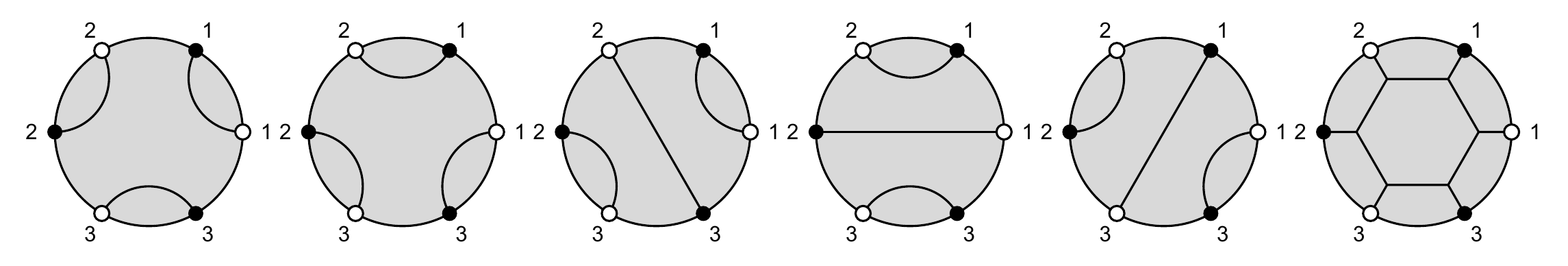}\end{center}
\caption{\label{Rt61}Classes of reduced webs with $6$ boundary points and $p=(w,b,w,b,w,b)$.}
\end{figure}

The matrix $\P$ is (rows are indexed by the webs in the figure; column labels correspond to the permutation
from white to black, in line notation, defined by pairing $\tau$.

\vspace{10pt}
\centerline{
\begin{tabular}{r@{$\,\,\,$}|c@{$\,\,\,$}c@{$\,\,\,$}c@{$\,\,\,$}c@{$\,\,\,$}c@{$\,\,\,$}c@{$\,\,\,$}}
       &123 & 132 & 213  & 231 & 312 & 321 \\
        \hline
$1$  &$1$&$0$&$0$&$0$&$0$&$0$ \\
$2$   &$0$&$0$&$0$&$0$&$1$&$0$\\
$3$   &$0$&$1$&$0$&$1$&$0$&$0$\\
$4$   &$0$&$0$&$1$&$1$&$0$&$0$\\
$5$   &$0$&$0$&$0$&$1$&$0$&$1$\\
$6$   &$0$&$0$&$0$&$-1$&$0$&$0$\\
\end{tabular}}
\vspace{7pt}

The corresponding polynomials are (recall that the coefficients in the table need to be multiplied by the signature of $\tau$)
$$\begin{pmatrix}P_1\\\vdots\\P_6\end{pmatrix} =\begin{pmatrix}
X_{1,1}X_{2,2}X_{3,3}\\
X_{1,3}X_{2,1}X_{3,2}\\
-X_{1,1}X_{2,3}X_{3,2}+X_{1,2}X_{2,3}X_{3,1}\\
-X_{1,2}X_{2,1}X_{3,3}+X_{1,2}X_{2,3}X_{3,1}\\
-X_{1,3}X_{2,2}X_{3,1}+X_{1,2}X_{2,3}X_{3,1}\\
-X_{1,2}X_{2,3}X_{3,1}\end{pmatrix}.$$

If $c\equiv 1$, note that (after multiplying the last one by $2$, its trace) their sum is $\det X = \frac{Z(c)}{\Delta^3}$.

In the scaling limit with (\ref{Xsclim}), and $w_1,b_1,w_2,b_2,w_3,b_3\to z_1,z_2,z_3,z_4,z_5,z_6$ we have limiting probabilities
\be\label{6probs}
\begin{pmatrix}Pr_1\\\vdots\\Pr_6\end{pmatrix}=\begin{pmatrix}
\frac{(z_3-z_2) (z_4-z_1)(z_5-z_2) (z_5-z_4)(z_6-z_1)(z_6-z_3)}{(z_3-z_1)(z_4-z_2) (z_5-z_3) (z_6-z_4)(z_5-z_1)(z_6-z_2)}\\
\frac{(z_2-z_1)(z_4-z_1) (z_4-z_3)(z_5-z_2) (z_6-z_3)(z_6-z_5)}{(z_3-z_1)(z_4-z_2) (z_5-z_3) (z_6-z_4)(z_5-z_1)(z_6-z_2)}\\   
\frac{(z_3-z_2)(z_4-z_3) (z_6-z_1)(z_6-z_5)}{(z_3-z_1)(z_5-z_3) (z_6-z_4)(z_6-z_2)}\\
\frac{(z_2-z_1)(z_4-z_3) (z_5-z_4)(z_6-z_1)}{(z_3-z_1)(z_4-z_2) (z_6-z_4)(z_5-z_1)}\\
\frac{(z_2-z_1)(z_3-z_2) (z_5-z_4)(z_6-z_5)}{(z_4-z_2)(z_5-z_3)(z_6-z_2)(z_5-z_1) }\\
\frac{2 (z_2-z_1)(z_3-z_2) (z_4-z_3)(z_5-z_4)(z_6-z_5) (z_6-z_1)}{(z_3-z_1)(z_4-z_2) (z_5-z_3) (z_6-z_4)(z_5-z_1)(z_6-z_2)}
\end{pmatrix}
\ee

This case is special in that the numerators and denominators factor into linear parts; this factorization does not hold for more general
boundary conditions.

\subsection{Type pattern $(b,b,w,b,w,w)$}

When $p=(b,b,w,b,w,w)$,
the six reduced web classes are shown in Figure \ref{Rt62}.
\begin{figure}[htbp]
\begin{center}\includegraphics[width=6.5in]{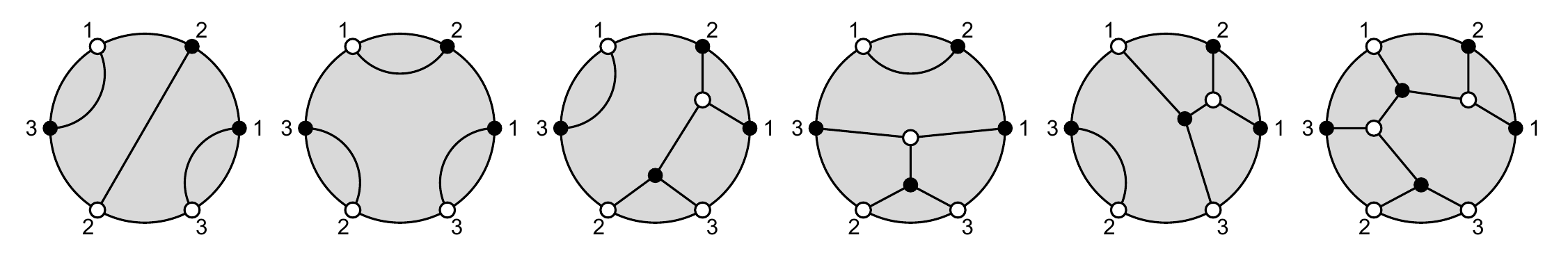}\end{center}
\caption{\label{Rt62}Classes of reduced webs with $6$ boundary points and $p=(b,b,w,b,w,w)$.}
\end{figure}

The matrix $\P$ is

\vspace{10pt}
\centerline{
\begin{tabular}{r@{$\,\,\,$}|c@{$\,\,\,$}c@{$\,\,\,$}c@{$\,\,\,$}c@{$\,\,\,$}c@{$\,\,\,$}c@{$\,\,\,$}}
       &123 & 132 & 213  & 231 & 312 & 321 \\
        \hline
$1$  &$0$&$0$&$0$&$0$&$1$&$1$ \\
$2$   &$1$&$1$&$1$&$1$&$0$&$0$\\
$3$   &$0$&$0$&$0$&$0$&$-1$&$0$\\
$4$   &$-1$&$0$&$-1$&$0$&$0$&$0$\\
$5$   &$-1$&$-1$&$0$&$0$&$0$&$0$\\
$6$   &$1$&$0$&$0$&$0$&$0$&$0$\\
\end{tabular}}
\vspace{7pt}

Consequently the polynomials are 
$$\begin{pmatrix}P_1\\\vdots\\P_6\end{pmatrix} =\begin{pmatrix}X_{1,3}X_{2,1}X_{3,2}-X_{1,3}X_{2,2}X_{3,1}\\
X_{1,1}X_{2,2}X_{3,3}-X_{1,1}X_{2,3}X_{3,2}-X_{1,2}X_{2,1}X_{3,3}+X_{1,2}X_{2,3}X_{3,1}\\
-X_{1,3}X_{2,1}X_{3,2}\\
-X_{1,1}X_{2,2}X_{3,3}+X_{1,2}X_{2,1}X_{3,3}\\
-X_{1,1}X_{2,2}X_{3,3}+X_{1,1}X_{2,3}X_{3,2}\\
X_{1,1}X_{2,2}X_{3,3}
\end{pmatrix}.$$

\subsection{Type pattern $(b,b,b,w,w,w)$}

This case is
shown in Figure \ref{Rt63}.
\begin{figure}[htbp]
\begin{center}\includegraphics[width=6.5in]{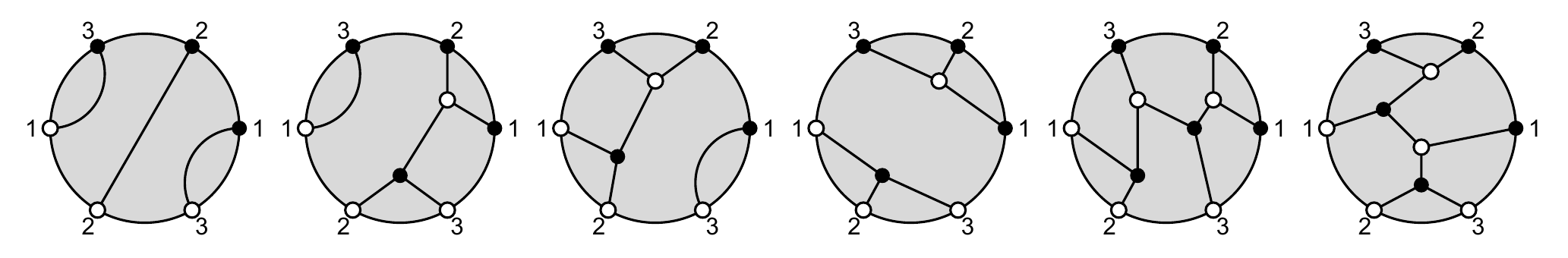}\end{center}
\caption{\label{Rt63}Classes of reduced webs with $6$ boundary points and $p=(b,b,b,w,w,w)$.}
\end{figure}
The matrix $\P$ is

\vspace{10pt}
\centerline{
\begin{tabular}{r@{$\,\,\,$}|c@{$\,\,\,$}c@{$\,\,\,$}c@{$\,\,\,$}c@{$\,\,\,$}c@{$\,\,\,$}c@{$\,\,\,$}}
        &123 & 132 & 213  & 231 & 312 & 321 \\
        \hline
$1$  &$1$&$1$&$1$&$1$&$1$&$1$ \\
$2$   &$-1$&$-1$&$-1$&$0$&$-1$&$0$\\
$3$   &$-1$&$-1$&$-1$&$-1$&$0$&$0$\\
$4$   &$-1$&$0$&$0$&$0$&$0$&$0$\\
$5$   &$1$&$1$&$0$&$0$&$0$&$0$\\
$6$   &$1$&$0$&$1$&$0$&$0$&$0$\\
\end{tabular}}
\vspace{7pt}

\subsection{Type pattern $(w,w,w,w,w,w)$}

In this case there are five classes, see Figure \ref{Rt64}.

\begin{figure}[htbp]
\begin{center}\includegraphics[width=5in]{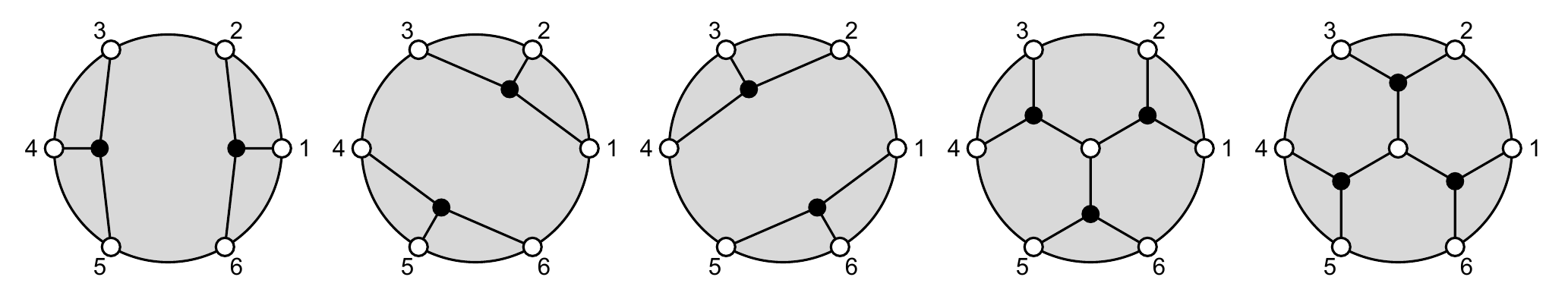}\end{center}
\caption{\label{Rt64}Classes of reduced webs with $6$ boundary points and $p=(w,w,w,w,w,w)$.}
\end{figure}
The matrix $\P$ is (with columns indexing the 3-partitions, enumerated by the part of the partition containing $1$)

\vspace{10pt}
\centerline{
\begin{tabular}{r@{$\,\,\,$}|c@{$\,\,\,$}c@{$\,\,\,$}c@{$\,\,\,$}c@{$\,\,\,$}c@{$\,\,\,$}c@{$\,\,\,$}c@{$\,\,\,$}c@{$\,\,\,$}c@{$\,\,\,$}c@{$\,\,\,$}}
        &123 & 124 & 125  & 126 & 134 & 135 & 136 & 145 & 146 & 156 \\
        \hline
$1$  &$0$&$0$&$1$&$1$&$0$&$1$&$1$&$0$&$0$&$0$ \\
$2$   &$1$&$1$&$0$&$0$&$0$&$1$&$0$&$1$&$0$&$0$\\
$3$   &$0$&$0$&$0$&$0$&$1$&$1$&$0$&$0$&$1$&$1$\\
$4$   &$0$&$0$&$0$&$0$&$0$&$-1$&$-1$&$-1$&$-1$&$0$\\
$5$   &$0$&$-1$&$-1$&$0$&$-1$&$-1$&$0$&$0$&$0$&$0$\\
\end{tabular}}
\vspace{7pt}

\bibliographystyle{alpha}
\bibliography{references.bib}

\end{document}